\documentclass[12pt]{article}

\usepackage
[colorlinks=true, pdfstartview=FitV, linkcolor=blue, citecolor=blue, urlcolor=blue]
{hyperref}

\usepackage{amssymb,amsmath, amscd,xypic, amsxtra}
\usepackage[T1]{fontenc}
\usepackage{times, verbatim}
\usepackage{graphicx}
\input xy
\xyoption{all}

\newtheorem{thm}{Theorem}[section]

\newtheorem{prop}[thm]{Proposition}

\newtheorem{lemma}[thm]{Lemma}

\newtheorem{cor}[thm]{Corollary}

\newcommand\proof
{\par\noindent{\bf Proof:\ } }
\newcommand\qed{\hfill$\blacksquare$}

\newcommand\cox{\mathsf{cox}}

\newcommand{\wh}{\widehat}
        
\DeclareMathOperator{\Ad}{Ad}

\DeclareMathOperator{\Irr}{Irr}

\DeclareMathOperator{\tr}{tr}

\DeclareMathOperator{\refl}{refl}

\DeclareMathOperator{\Sym}{Sym}
\DeclareMathOperator{\triv}{triv}

\newcommand{\bn}{\mathbb{N}}

\newcommand{\bc}{\mathbb{C}}
\newcommand{\bF}{\mathbb{F}}

\newcommand{\bq}{\mathbb{Q}}
\newcommand{\bz}{\mathbb{Z}}

\newcommand{\al}{\alpha}
\newcommand{\be}{\beta}
\newcommand{\ga}{\gamma}
\newcommand{\de}{\delta}
\newcommand{\De}{\Delta}
\newcommand{\Ga}{\Gamma}
\newcommand{\lam}{\lambda}

\newcommand{\om}{\omega}

\newcommand{\ep}{\epsilon}
\newcommand{\vep}{\varepsilon}
\newcommand{\vp}{\varphi}

\newcommand{\vt}{\vartheta}

\newcommand{\scb}{\mathcal{B}}

\newcommand{\scd}{\mathcal{D}}

\newcommand{\sch}{\mathcal{H}}

\newcommand{\scp}{\mathcal{P}}

\newcommand{\sct}{\mathcal{T}}

\newcommand{\la}{\langle}
\newcommand{\ra}{\rangle}
\newcommand{\lra}{\longrightarrow}

\newcommand{\hra}{\hookrightarrow}

\DeclareMathOperator{\diag}{diag}
\DeclareMathOperator{\GL}{GL}

\DeclareMathOperator{\SL}{SL}

\DeclareMathOperator{\SO}{SO}
\DeclareMathOperator{\SU}{SU}

\DeclareMathOperator{\Reg}{Reg}
\DeclareMathOperator{\Sp}{Sp}
\DeclareMathOperator{\Spin}{Spin}

\newcommand{\fg}{\mathfrak{g}}

\newcommand{\ft}{\mathfrak{t}}

\newcommand{\gint}[2]{
\left\lfloor{\frac{#1}{#2}}\right\rfloor}

\newcommand{\lint}[2]{
\left\lceil{\frac{#1}{#2}}\right\rceil}

\newcommand{\D}[4]{
{\scriptsize\begin{matrix}
#1&#2&#3\\
&#4&
\end{matrix}} }

\newcommand{\E}[6]{
{\scriptsize\begin{matrix}
#1\!\!&\!\!#2\!\!\!&\!\!\!#3\!\!\!&\!\!\!#4\!\!&\!\!#5\\
\!\!\!&\!\!\!&#6&\!\!\!&\!\!\!
\end{matrix}} }

\begin{document}

\title {Weyl group characters afforded by zero weight spaces}

\author{Mark Reeder}

\maketitle

\begin{center}
{\it In memory of Bert Kostant}
\end{center}

\begin{abstract} 
Let $G$ be a compact Lie group with Weyl group $W$. We give a formula for the character of $W$ on the zero weight space of any finite dimensional representation of $G$. The formula involves weighted partition functions, generalizing Kostant's partition function. On the elliptic set of $W$ the partition functions are trivial. On the elliptic regular set, the character formula is a monomial product of certain co-roots, up to a constant equal to $0$ or $\pm 1$. This generalizes Kostant's formula for the trace of a Coxeter element on a zero weight space. 
If the long element $w_0=-1$, our formula gives a method for determining all representations of $G$ for which the zero weight space is irreducible. \end{abstract}

\section{Introduction and statement of results} Let $G$ be a compact Lie group with maximal torus $T\subset G$. 
One of the oldest problems in representation theory is to decompose the representation of $G$
on the Hilbert space $L^2(G/T)$ of functions on $G/T$ which are square-integrable with respect to a $G$-invariant measure.  From the Peter-Weyl theorem  it follows that an irreducible  representation $V$ of $G$ appears in $L^2(G/T)$ with multiplicity equal to the dimension of the space $ V^T$ of $T$-invariant vectors in $V$. Kostant showed that $\dim V^T$ can be expressed in terms of his {\it Partition Function}, which counts the number of ways a weight can be expressed as a non-negative linear combination of positive roots.  Therefore the $G$-decomposition of $L^2(G/T)$ is known, in principle (see section \ref{explicit} below).  

It is also natural to ask for the $G$-decomposition of $L^2(\sct)$, where $\sct$ is the homogenous space of all maximal tori in $G$. This problem has not been solved, even in principle, except for $\SU_2$ (well-known, see below) and $\SU_3$ (following from known results, see section \ref{SU3}). 

More generally, on $G/T$ there is also a right action by the Weyl group $W$ commuting with the left $G$-action, so $L^2(G/T)$ is actually a representation of  $G\times W$. 
If $V\in \Irr(G)$ and $U\in\Irr(W)$ then the multiplicity of $V\boxtimes U$ in $L^2(G/T)$ equals the multiplicity of $U$ in $V^T$. For example, the multiplicity of $V$ in $L^2(\sct)$ equals the multiplicity of the trivial character of $W$ in $V^T$.

If $G=\SU_2$ then $G/T=S^2$ and the nontrivial element of $W$ acts on $S^2$ via the antipodal map. The irreducible constituents of $L^2(S^2)$ are $V_\mu\boxtimes \vep^m$ where $V_\mu$ has odd dimension $\mu=2m+1$, and $\vep$ is the nontrivial character of $W$. Also, $\sct$ is the real projective plane and 
$V_\mu$ appears in $L^2(\sct)$ with multiplicity one if $\mu\in 1+4\bz$, zero otherwise.

For larger groups, the first results on the $W$-decomposition of $V^T$ were obtained in the 1970's by Gutkin \cite{gutkin} and Kostant \cite{kostant:eta}. Much work has been done since (see section \ref{earlier}), but until now  the character of $V^T$, for general $V\in\Irr(G)$, was known on just one non-identity conjugacy class in $W$, namely the class $\cox$ consisting of the Coxeter elements. This result was also obtained by Kostant. He showed that $\tr(\cox,V^T)\in\{-1,0,1\}$ and he gave a formula for the exact value, in terms of the $W$-action on a certain finite quotient of the character lattice of $T$.

In this paper we give, for any $G$, any $V\in \Irr(G)$ and any $w\in W$ a formula for the character  $\tr(w,V^T)$, in terms of the highest weight of $V$. For those $w$ sharing certain properties with those of $\cox$, we give a direct generalization of Kostant's formula for $\tr(\cox,V^T)$ in terms of the $W$-action on other finite quotients of the character lattice of $T$.

To explain our formulas, it is useful to rank the elements of $W$ according to 
the dimension $d(w)$ of the fixed-point set of $w$ in $T$. At one extreme, the identity element $1_W$ has $d(1_W)=\dim T$ and Kostant's partition function gives a formula for  $\tr(1_W,V^T)$, in principle. At the other extreme, $d(\cox)=0$ and we have seen there is a simple formula for $\tr(\cox,V^T)$. Thus we expect $d(w)$ to measure the complexity of $\tr(w,V^T)$ as a function of the highest weight of $V$. Indeed, our formula for $\tr(w,V^T)$ interpolates between Kostant's formulas for $\tr(1_W,V^T)$ and $\tr(\cox, V^T)$ and involves a weighted partition function $\scp_w$ of rank equal to $d(w)$. 

\subsection{Some basic notation}

We assume $G$ is connected and all normal subgroups of $G$ are finite. For technical reasons we also assume $G$ is simply connected, even though any $V$ with $V^T\neq 0$ factors through the quotient of $G$ by its center.
Let $\fg$ and $\ft$ be the complexified Lie algebras of $G$ and $T$ and let 
$R$ be the set of roots of $T$ in $\fg$. Choose a set $R^+$ of positive roots in $R$, and let $\rho$ be half the sum of the roots in $R^+$. Let $P$ be the additive group indexing the characters of $T$; we write $e_\lam:T\to S^1$ for the character indexed by $\lam\in P$. Let $P_{++}$ be the set of dominant regular characters of $T$ with respect to $R^+$.
For $\mu\in P_{++}$, let  $V_\mu$ be the irreducible representation of $G$ with highest weight $\mu-\rho$. (This is consistent with the $\SU_2$ example above, and greatly simplifies our formulas.) The sign character of $W$ is denoted by $\vep$.

\subsection{Elliptic traces on zero-weight spaces} \label{intro:elltrace} 

When $d(w)=0$ we say  $w$ is {\it elliptic}. The character of $V^T$ on the elliptic set in $W$ has  intrinsic meaning: It determines the virtual character of $V^T$ modulo linear combinations of induced representations from proper parabolic subgroups of $W$ \cite{reeder:elliptic}. Coxeter elements are elliptic; in $\SU_n$ there are no others. For all $G\not\simeq \SU_n$ there are non-Coxeter elliptic elements.  For example $W(E_8)$ has 30 elliptic conjugacy classes. 

Assume $w\in W$ is elliptic.
We regard $w$ as a coset of $T$ in the normalizer $N_G(T)$. 
Because $w$ is elliptic, the elements of $w$ are contained in a single conjugacy class $w^G$ in $G$ and we have
\begin{equation}\label{wcfelliptic}
\tr(w,V_\mu^T)=\tr(t,V_\mu),
\end{equation}
for any element $t\in T\cap w^G$. 
Therefore $\tr(w,V_\mu^T)$ can be computed from the Weyl Character Formula, after cancelling poles arising from the set $R_t^+=\{\al\in R^+:\ e_\al(t)=1\}$. We obtain (\cite{reeder:prehom} and section \ref{weylharmonic} below)
\begin{equation}\label{intro:elliptictrace}
\tr(w,V_\mu^T)=\frac{1}{\De(t)}\sum_{v\in W^t} \vep(v)e_{v\mu}(t) H_t(v\mu),
\end{equation}
where 
\[\De=
e_\rho\prod_{\al\in R^+\setminus R_t^+}(1-e_{-\al}(t)),\quad
W^t=\{v\in W: v^{-1}R_t^+\subset R^+\},\quad
H_t(v\mu)=\prod_{\al\in R_t^+}
\frac{\la v\mu,\check\al\ra}{\la \rho_t,\check\al\ra}
\]
and $\rho_t$ is the half-sum of the roots in $ R_t^+$.  
The monomials $H_t$ are $W$-harmonic polynomials on $\ft^\ast$. Their geometric meaning is discussed at the end of section \ref{weylharmonic}.  

Let $\wh G$ be a compact Lie group dual to $G$. Then $P$ may be regarded as the lattice of one-parameter subgroups of a maximal torus $\wh T$ of $\wh G$. Let $m$ be the order of $w$ and let $\hat\mu=\mu(e^{2\pi i/m})\in \wh T$. Let $C_G(t)$ and $C_{\wh G}(\wh\mu)$ be the corresponding centralizers. From \eqref{intro:elliptictrace}, one gets the following vanishing result.

\begin{thm}\label{ellvanishing}
If $\dim C_G(t)<\dim C_{\wh G}(\wh\mu)$ then 
$\tr(w,V_\mu^T)=0$. 
\end{thm}

For example, if $\mu\in mQ$, where $m$ is the order of $w$, then $\wh \mu=1$, so Thm. \ref{ellvanishing} implies that  $\tr(w,V_\mu^T)=0$. Hence if 
$\mu\in nQ$, where $n$ is the least common multiple of the orders of the elliptic elements in $W$ then Thm. \ref{ellvanishing} implies that the character of $V_\mu^T$ is a linear combination of induced characters from proper parabolic subgroups of $W$. 

Though useful, formula \eqref{intro:elliptictrace} does not tell the whole story. For example if $w=\cox$ it expresses $\tr(\cox, V_\mu^T)$ as a sum of $|W|$ terms, but we know from Kostant that cancellations put the actual trace in $\{-1,0,1\}$. 
This is because Coxeter elements have an additional property shared by some but not all elliptic elements.

We say that $w\in W$ is {\it regular} if the subgroup of $W$ generated by $w$ acts freely on $R$ (cf. \cite{springer:regular}). For elliptic regular elements we generalize Kostant's formula as follows. 
\begin{thm}\label{ellreg}
Assume $w$ is elliptic and regular. Then 
$\tr(w,V_{\mu}^T)=0$ unless there exists $v\in W$ such that $v\mu\in \rho+mQ$, in which case
\[
\tr(w,V_{\mu}^T)=\vep(v)\prod
\frac{\la \mu,\check\al\ra}{\la \rho,\check\al\ra},
\]
where the product is over the positive coroots $\check\al$ of $G$ for which $\la \mu,\check\al\ra\in m\bz$. 
\end{thm}
Theorem \ref{ellreg} shows that the harmonic polynomial in \eqref{intro:elliptictrace} is actually a harmonic {\it monomial} when viewed from the dual group $\wh G$.
The key to Thm. \ref{ellreg} is that the assumed inequality in Thm. \ref{ellvanishing} holds automatically when (and only when) $w$ is both elliptic and regular. (See \cite{reeder:thomae} and section \ref{sec:ellreg} below.)

The dual group $\wh G$ was first used by D. Prasad in \cite{prasad:cox} to give a new interpretation of Kostant's result for $\tr(\cox, V_\mu^T)$.
In this case $m=h$ is the Coxeter number,  $\check R^+_\mu=\varnothing$ and Thm. \ref{ellreg} becomes Kostant's formula for $\tr(\cox, V_\mu^T)$.  

At the opposite extreme, if the long element $w_0=-1$ in $W$ then $w_0$ is an elliptic involution. Applying Thm. \ref{ellreg} to $w_0$ gives the following qualitative result (see section \ref{irredzero}). 

\begin{cor}\label{irred} Assume  $-1\in W$ but $G$ is not $\SU_2$.  Then there are only finitely many $V\in \Irr(G)$ for which $V^T$ is an irreducible representation of $W$. 
\end{cor} 

For $\SU_n$, $n\geq 2$, the symmetric powers $\Sym^{kn}(\bc^n)^T$, for $k=0,1,2,\dots$, afford the trivial and sign and sign characters alternately. On the other hand, for $\SU_4$ the two-dimensional representation of $W$ is all of $V_\mu^T$ for just one $\mu$ (see section \ref{SU4}).
Cor. \ref{irred} can be sharpened to classify irreducible zero weight spaces of other groups. See section \ref{irredzero} which includes some history of this problem.

\subsection{A general character formula for zero weight spaces}

Now take any element $w\in W$. Let $S=(T_w)^\circ$ be the identity component of $T_w$ and let $L=C_G(S)$ be the centralizer of the torus $S$.  There is an element $t\in T$ which is $L$-conjugate to an element of the coset $w$. We fix such a $t$ and consider the coset $tS\subset T$. 

The positive roots are partitioned as  $R^+=R_{tS}^+\sqcup R^{tS}_1\sqcup R^{tS}_2$, 
where 
\[\begin{split}
R_{tS}^+&=\{\al\in R^+:\ e_\al\equiv 1\ \text{on $tS$}\}, \\
R^{tS}_1&=\{\al\in R^+:\ e_\al\equiv e_\al(t)\neq 1\ \text{on $tS$}\},\\ 
R^{tS}_2&=\{\al\in R^+:\ e_\al \ \text{is nonconstant on $tS$}\}.
\end{split}
\]
As in the elliptic case, we set 
\[\begin{split}
\rho_{tS}&=\frac{1}{2}\sum\limits_{\al\in R_{tS}^+}\al\qquad \qquad\ 
\De=e_\rho\prod_{\al\in R^{tS}_1}(1-e_{-\al})\\
H_{tS}(\lam)&=\prod\limits_{\al\in R_{tS}^+}\dfrac{\la \lam,\check\al\ra}{\la\rho_{tS},\check\al\ra}
\qquad W^{tS}=\{v\in W:\ v^{-1}R_{tS}^+\subset R^+\}.
\end{split}
\]
The set $R^{tS}_2=\{\be_1,\dots,\be_r\}$ determines a weighted partition function $\scp_w$ as follows. 
Let $Y$ be the character lattice of $S$. For each $i=1,\dots,r$, the restriction of $e_{\be_i}$ to $S$ is a non-trivial character $e_{\nu_i}$ for some $\nu_i\in Y$. Let $z_i=e_{-\be_i}(t)$. 
Now for $\nu\in Y$, let 
\[\scp_w(\nu)=\sum z_1^{n_1}z_2^{n_2}\cdots z_r^{n_r}\]
where the sum runs over all $r$-tuples $(n_1,\dots, n_r)$ of nonnegative integers such that 
$\sum\limits_{i=1}^rn_i\nu_i=\nu$. 
\begin{thm}\label{general}
With notation as above, we have 
\begin{equation}\label{eq:general}\tr(w,V_\mu^T)=
\frac{1}{\De(t)}\sum_{v\in W^{tS}}\vep(v) e_{v\mu}(t)\scp_w(v\mu-\rho)H_{tS}(v\mu).
\end{equation}
\end{thm}
This is proved similarly to \eqref{intro:elliptictrace}, but now we take the constant term along $S$ of the restriction of the character of $V_\mu$ to the coset $tS$. See section \ref{sec:general}.

If $w=1_W$, formula \eqref{eq:general} becomes Kostant's formula for $\dim V_\mu^T$. Some examples of intermediate partition functions $\scp_w$ are found in  section \ref{smallgroups}.

\subsection{Explicit results}\label{explicit} For non-elliptic $w$, the formula \eqref{eq:general} is  difficult to compute explicitly as a function of $\mu$. The difficulty, measured by $d(w)=\dim S$, is in the weighted partition function $\scp_w$. The most difficult case is $\dim V_\mu^T$, which is known explicitly only for small groups. For general $G$, it was shown in \cite{kumar-prasad}  that the function $\mu\mapsto\dim V_\mu^T$ is a piecewise polynomial function, but there are no explicit formulas for these polynomials. 
For $w\neq 1$, the function $\mu\mapsto\tr(w,V_\mu^T)$ also appears to be a piecewise polynomial function.

In section \ref{smallgroups} we give explicit formulas for $\tr(w,V_\mu^T)$ for all $w$ in the groups of rank two and also $\SU_4$; these are the groups for which explicit polynomial formulas for $\dim(V_\mu^T)$ are known (to me). 

For $\Spin_8$ and $F_4$ we carry out the idea of \ref{irred} above find all $\mu$ for which $V_\mu^T$ is irreducible. These turn out to be just the known examples, coming from small representations \cite{reeder:zero1}.

For $E_7$ and $E_8$ the same idea, with more computation, should also find all irreducible $V_\mu^T$, but we do not address this here.  

For $E_6$, $-1\notin W$. Instead we explicitly compute 
$\tr(w, V_\mu^T)$ for the elliptic triality $w$, using  Thm. \ref{ellreg}. 
When $\tr(w, V_\mu^T)\neq 0$ the method of Cor. \ref{irred} applies. We find that the only  
possibilities for an irreducible representation of $W$ to be a zero weight space are the five known ones (if $\tr(w, V_\mu^T)\neq 0$) and two possible additions (if $\tr(w, V_\mu^T)= 0$). See section \ref{E6}.

\subsection{Earlier work}\label{earlier} Zero-weight spaces have been much-studied in the past half-century. A recent survey of the problem is given in \cite{humphreys}.

For classical groups, \cite{ariki}, \cite{aik}, \cite{amt}, \cite{gay}, \cite{gutkin} use Schur-Weyl duality express the decomposition of $V_\mu^T$ in terms of induced representations of symmetric groups.  For a fixed classical group, say $G=\SU_3$, this method involves multiplicities in induced representations of arbitrarily large symmetric groups. 

For $G=G_2$, \cite{mavrides} gives a computer algorithm, but not a formula, for explicitly computing the character of $V_\mu^T$ for any given $\mu$. 

For ``small" $V_\mu$, the decomposition of $V_\mu^T$ is given explicitly in \cite{gutkin} (for $\SL_n$), \cite{reeder:zero1} (for types $D_n$ and $E_n$) and \cite{reeder:zero2} (for types $B_n, C_n, G_2, F_4$). 

\subsection{Organization of the paper}

Section \ref{sec:weyltorsion} is purely about the Weyl character formula on torsion elements of $G$. These results are applied to elliptic traces in  section \ref{sec:elltrace}. In section \ref{sec:wcf} we return to the Weyl character formula, now to study its values on cosets of subtori in $T$.  
These results are applied to the traces of general elements $w\in W$ in   section \ref{sec:generalw}.
In section \ref{sec:smallgroups} we give explicit formulas for the character of $V_\mu^T$,  for small groups. The final section \ref{sec:irredzero} concerns irreducible zero weight spaces. 

\tableofcontents

\section{The Weyl character formula and torsion elements }
\label{sec:weyltorsion}

This section is purely about the Weyl Character formula. 
Notation is that of section \ref{intro:elltrace}. 

\subsection{The Weyl character formula and harmonic polynomials}\label{weylharmonic}

It is known, from  \cite{reeder:prehom} or as a special case of Thm. \ref{mainthm} below, that for $\mu\in P_{++}$ we have 
\begin{equation}\label{tr(t)}
\tr(t,V_\mu)=
\frac{1}{\De(t)}\sum_{v\in W^t}\vep(v)e_{v\mu}(t)H_t(v\mu).
\end{equation}

The function $\mu\mapsto H_t(\mu)$ and its $W$-translates $H_t^v(\mu):=H_t(v\mu)$ belong to the vector space $\sch$ of $W$-harmonic polynomials on
 $\ft^\ast$.  In \cite{macdonald:truncated} Macdonald showed that $\{H_t^v:\ v\in W^t\}$ spans an irreducible representation of $W$, now called the {\it truncated induction} of the sign character of $W_t$ to $W$. Here $W_t$ is the Weyl group of the  centralizer $C_G(t)$. 

By a theorem of Borel (see \cite[section 5]{reeder:cohom} for a simple proof), the graded vector space $\sch$ is canonically isomorphic to the homology of the flag manifold $\scb=G/T$. The translates $H_t^v$, for $v\in W^t$, correspond to the fundamental classes  of the connected components of the fixed-point submanifold $\scb_t\subset\scb$. 
In this interpretation, equation \eqref{tr(t)}, combined with the Borel-Weil theorem, is a special case of the Atiyah-Segal fixed-point theorem \cite[Theorem (3.3)]{atiyah-segal}.  However the proof of \eqref{tr(t)} given in \cite{reeder:prehom} or Thm. \ref{mainthm} below, is elementary; it boils down to the Weyl dimension formula for $C_G(t)$.

\subsection{Traces of torsion elements in $G$} \label{sec:torsiontrace}

In this section we study the values of \eqref{tr(t)} on torsion elements of $G$.

\begin{prop}\label{torsiontrace} Let $m$ be a positive integer. Suppose $t\in T$ has $\Ad(t)$ of order $m$, and let $y$ be a coset of $mQ$ in $P$. 
\begin{itemize}
\item[(a)] If $|R_t^+|<|\check R_y^+|$, then $\tr(t,V_\mu)=0$ for all  $\mu\in y\cap P_{++}$.  
\item[(b)] If $|R_t^+|=
|\check R_y^+|$ then there is a constant $C_{t,y}\in\bc$ such that 
\[\tr(t,V_\mu)=C_{t,y}\prod_{\check\al\in\check R_y^+}\la\mu,\check \al\ra,
\]
for all $\mu\in y\cap P_{++}$ 
\end{itemize}
\end{prop}
\begin{proof}
For any $v\in W$,  the function $\mu\mapsto e_{v\mu}(t)$ is constant on each coset of $mQ$ in $P$. 
Hence, for each coset $y\in P/mQ$, we can define a function $\tau_y:W\to\bc$ by 
$\tau_y(v)=e_{v\mu}(t) $ for any $\mu\in y$. We set
\[K_y=\frac{1}{|W_t|}\sum_{v\in W}\vep(v)\tau_y(v) H_t^v,\]
where $H_t^v(\mu)=H_t(v\mu)$. 
Thus $K_y$ is a harmonic polynomial on $\ft^\ast$, of degree $|R_t^+|$ and depending only on the coset $y\in P/mQ$. From equation \eqref{tr(t)} we have
\begin{equation}\label{trQ}
\tr(t,V_\mu)=\frac{1}{\De(t)}\cdot K_y(\mu),\qquad \text{for all}\ \ \mu\in y.
\end{equation}
One checks that 
\begin{equation}\label{dey}
\tau_y(vu)=\tau_y(v)\quad\text{ for all}\quad u\in W_y. 
\end{equation}
Equation \eqref{dey} implies that under the action of $W$ on $\sch$ we have
\[K_y^u=\vep(u)K_y,\quad\text{ for all}\quad u\in W_y.\]
It follows that $K_y$ is divisible by the polynomial
\begin{equation}\label{Cty}
\prod_{\check\al\in\check R^+_y}\check \al.
\end{equation}
Hence if $|R_t^+|<|\check R_y^+|$ we have $K_y=0$ and if 
$|R_t^+|=|\check R_y^+|$ then $K_y$ is a constant multiple of the polynomial \eqref{Cty}.
This proves the Proposition.
\end{proof}

\subsection{Principal weights}

We  continue with $t\in T$ having $\Ad(t)$ of order $m$. Now we consider certain families of weights.

An  weight $\mu\in P$ is {\it $m$-principal} if the $W$-orbit of $\mu$ meets $\rho+mQ$. If $y\in P/mQ$ is the coset of an $m$-principal weight we can compute the constant $C_{t,y}$ appearing in Proposition \ref{torsiontrace}, as follows. 

\begin{lemma}\label{yrho} Assume $y=v\rho+mQ$ for some $v\in W$. Then we have
\[K_{y}(v\rho)=\vep(v)\De(t).\]
\end{lemma}
\proof One checks that for any $v\in W$ we have 
\[K_{v\rho+mQ}(v\rho)=\vep(v)K_{\rho+mQ}(\rho),\]
so we may assume $v=1$.   

From Weyl's identity (cf. \eqref{D} below) applied to $W_{t}$ 
it follows that 
\[
\De(t)=
\sum_{v\in W^t}\vep(v)e_{v\rho}(t)H(v\rho)=K_{\rho+mQ}(\rho), 
\]
as desired. 
\qed

Set 
\[\check R_m^+=\{\check\al\in \check R^+:\ \la \rho,\check\al\ra =m\}
\]

\begin{prop}\label{Requal} Let $t\in T$ have $\Ad(t)$ of order $m$. Assume that 
\begin{itemize}
\item[\text(a)] $|R_t^+|=|\check R_m^+|$, and 
\item[\text(b)] $\mu\in P_{++}$ is $m$-principal.
\end{itemize} Then 
\begin{equation}\label{Requal2}
\tr(t,V_\mu)=\vep(v)\prod_{\check\al\in \check R_m^+}
\frac{\la v \mu,\check\al\ra}{\la \rho,\check\al\ra},
\end{equation}
where $v$ is any element of $W$ for which $v\mu\in \rho+mQ$. 
\end{prop}
\proof 
From equation \eqref{trQ} we have 
\[\tr(t,V_\mu)=\frac{K_y(\mu)}{ \De(t)}.\]
For each $\al\in\check R_y^+$ let $\ep_\al=\pm 1$ be such that $\ep_\al \cdot v\al\in \check R^+$. 
Sending $\al\mapsto \ep_\al \cdot v\al$ is a bijection 
$\check R_y^+\to \check R_m^+$, so we have 
$|R_t^+|=|\check R_y^+|$. Applying Lemma \ref{yrho} and equation \eqref{Cty}, we get 
\[
\vep(v)\De(t)=K_y(v\rho)
=C_{t,y}\prod_{\check\al\in \check R_y^+}
\la v^{-1} \rho,\check\al\ra=
C_{t,y}\prod_{\check\al\in \check R_y^+}
\la \rho,v\check\al\ra,
\]
so 
\[C_{t,y}=\frac{\vep(v)\De(t)}{\prod\limits_{\check\al\in \check R_y}\la \rho,v\check\al\ra}.\]
It follows that 
\[\begin{split}
\tr(t,V_\mu)=\vep(v)\prod_{\al\in\check R_y^+}
\frac{\la \mu,\check\al\ra}{\la \rho,v\check\al\ra}
=\vep(v)\prod_{\al\in\check R_y^+}
\frac{\la v\mu,\ep_\al \cdot v\check\al\ra}{\la \rho,\ep_\al \cdot v\check\al\ra}
=\vep(v)\prod_{\al\in\check R_m^+}
\frac{\la v\mu,\check\al\ra}{\la\rho,\check\al\ra}.
\end{split}
\]
\qed

{\bf Remark.\ } From 
Lemma \ref{Wstab} below, the stabilizer of $\rho+mQ$ in $P/mQ$ is the reflection subgroup  $W_m=\la r_\al:\ \check\al\in R_m^+\ra$. 
Since the polynomial
\[\prod_{\check\al\in R_m^+}\check\al\]
transforms under the sign character of $W_m$, it follows that the right side of \eqref{Requal2} is independent of the choice of $v$, as it must be.

The set
\[T_m:=\{t\in T:\ \text{$\Ad(t)$ has order $m$ and $|R_t|=|\check R_m|$}\}
\] 
may consist of several $W$-orbits. These must be separated by characters of $G$. However, Prop. \ref{Requal} implies that if $\mu\in P_{++}$ is $m$-principal then the character of $V_\mu$ is constant on $T_m$. 

Certain elements of $T_m$ are obtained as follows. Let $2\check\rho$ be the sum of the coroots $\check\al\in\check R^+$. If $\eta\in S^1$ has order $2m$, 
then the element $t=2\check \rho(\eta)$ belongs to $T_m$.  For such elements, Prop. \ref{Requal} gives the following character value. 

\begin{cor}\label{principal} \label{bothprincipal}
Let $m>1$ be an integer. Assume that $\mu+mQ$ is $m$-principal and that $t=2\check\rho(\eta)$, where $\eta\in S^1$ has order $2m$. Then we have 
\[\tr(t,V_\mu)=\vep(v)\prod_{\check\al\in \check R_m^+}
\frac{\la v\mu,\check\al\ra}{\la \rho,\check\al\ra},
\]
where $v\mu\in \rho+mQ$.
\end{cor}

\subsection{Dual groups and torsion elements} \label{dualtorsion}
It will be helpful to study the $W$-action on $P/mQ$ in terms of the dual group of $G$. 

The simply connected group  $G$ has root datum 
\[\scd(G)=(P,R,\check Q,\check R).\]
Let  $\wh G$ be a simple compact Lie group with dual root datum
\[\scd(\wh G)=(\check Q,\check R,P,R).\]
We fix a maximal torus $\wh T$ in $\wh G$ and identify $P$ with the lattice of one-parameter subgroups of $\wh T$.  

Let $\pi:\wh G_{sc}\to \wh G$ be the simply connected covering map. 
The group $\wh G_{sc}$ has root datum 
\[\scd(\wh G_{sc})=(\check P,\check R,Q,R),\]
where $\check P$ is the lattice of one-parameter subgroups in a maximal torus of the adjoint group of $G$. The kernel of $\pi$ is the center $\wh Z_{sc}$ of $\wh G_{sc}$ and we have a canonical isomorphism $\wh Z_{sc}\simeq P/Q$. The preimage torus $\wh T_{sc}=\pi^{-1}(\wh T)$ may be identified with 
 $Q\otimes S^1$ and the
exact sequence
 \begin{equation}\label{hatGsc}
 1\lra \wh Z_{sc}\lra \wh G_{sc}\overset{\pi}\lra \wh G\lra 1
 \end{equation}
restricts to a $W$-equivariant exact sequence on maximal tori:
 \begin{equation}\label{hatTsc}
 1\lra \wh Z_{sc}\lra \wh T_{sc}\overset{\pi}\lra\wh T\lra 1.
 \end{equation}
 
Now let $\zeta\in S^1$ have finite  order $m$. 
 Evaluation at $\zeta$ gives an isomorphism from  $P/mP$ to the torsion subgroup $\wh T[m]:=\{\tau\in \wh T:\ \tau^m=1\}$. The preimage
 $\wh T_{sc}[m]:=\pi^{-1}(\wh T[m])$ fits into
 the exact sequence 
\begin{equation}\label{PQT}
1\lra\wh Z_{sc}\lra \pi^{-1}(\wh T[m])\overset{\pi}{\lra} \wh T[m]\lra 1.
\end{equation}
This sequence is $W$-equivariantly isomorphic to the canonical exact sequence 
\begin{equation}\label{PQ}
0\lra P/Q\overset{m}{\lra} P/mQ\lra P/mP\lra 0.
\end{equation}
For $\mu\in P$, let $W_\mu=W_y$ be the stabilizer of the coset $y=\mu+mQ$ under the $W$-action on $P/mQ$, 
and let $\check R_\mu^+=\{\check\al\in \check R^+:\ \la\mu,\check\al\ra\in m\bz\}$.  This set of roots depends only on the coset $y=\mu+mQ$. We sometimes write 
\[\check R_y^+=\check R_\mu^+.\]

Let $\zeta\in\bc^\times$ have order $m$ as above. Then $\check R_\mu^+$ is a set of positive roots for he connected centralizer $C_{\wh G}^\circ(\mu(\zeta))$ and $W_\mu$ is the Weyl group of $C_{\wh G}^\circ(\mu(\zeta))$.

\begin{lemma}\label{Wstab} The group $W_\mu$ is generated by $\{r_\al: \check\al\in \check R_\mu^+\}$. 
\end{lemma}
\proof This follows from the isomorphism $P/mQ\simeq \pi^{-1}(\wh T[m])$ and the
fact that $\wh G_{sc}$ is simply connected. 
\qed

\section{Elliptic traces on zero weight spaces}\label{sec:elltrace}
We apply the results of the previous section to compute 
$\tr(w,V_\mu^T)$ when $w\in W$ is elliptic.  Ellipticity of $w$ is equivalent to any of the following:  
$T_w$ is finite;  $w$ fixes no nonzero element of $P$; all elements of the coset $w\subset N$ are $T$-conjugate; $w$ itself is contained in a single $G$-conjugacy class. 
The unique $G$-conjugacy class containing $w$  is denoted  by $w^G$.

One consequence of ellipticity is the following. 
\begin{lemma}\label{transitive} Assume $w\in W$ is elliptic.
 Let $t\in T\cap w^G$ and let $z$ belong to the center $Z$ of $G$. Then there exists $v\in W$ such that $t^v=tz$. 
 \end{lemma}
 \proof Since $t\in w^G$, there exist $g\in G$ and $n\in w$ such that $t=gng^{-1}$. 
 Since $z\in Z$ we also have $zt=gzng^{-1}$. 
But $zn\in w$ and $w$ is elliptic so there exists $s\in T$ such that 
$zn=sns^{-1}$. Hence $zt=gsns^{-1}g^{-1}=hth^{-1}$, where $h=gsg^{-1}$, 
so $zt$ and $t$ are $G$-conjugate elements of $T$. It follows that $zt$ and $t$ are $W$-conjugate, as claimed.
\qed

Another consequence is that, since an elliptic element $w\in W$ fixes no nonzero weight in $P$, we have
\[\tr(w,V_\mu^T)=\tr(w^G,V_\mu).\]
From equation \eqref{tr(t)} it follows that  \begin{equation}\label{elliptictrace}
\tr(w,V_\mu^T)=\frac{1}{\De(t)}\sum_{v\in W^1}\vep(v)e_{v\mu}(t)H(v\mu).
\end{equation}
for any $t\in T\cap w^G$. 

By \cite[Theorem 8.5]{springer:regular}, the character values  $\tr(w,V_\mu^T)$ are rational for all dominant regular weights $\mu$. 
From Proposition \ref{torsiontrace} we obtain

\begin{prop}\label{elltrace} Assume $w\in W$ is elliptic. Let $t\in T\cap w^G$ have $\Ad(t)$ of order $m$ and let $y\in P/mQ$.
\begin{itemize}
\item[(a)] If $|R_t|<|\check R_y|$ then $\tr(w,V_\mu^T)=0$ for all  $\mu\in y\cap P_{++}$.  
\item[(b)] If $|R_t|=|\check R_y|$ then there is a constant $C_{t,y}\in\bq$ such that for all $\mu\in y\cap P_{++}$ we have
\[\tr(w,V_\mu^T)=C_{t,y}\cdot\prod_{\check\al\in\check R_y}\la\mu,\check \al\ra.
\]
\end{itemize}
\end{prop}

\subsection{Elliptic regular traces}\label{sec:ellreg}

We now assume $w\in W$ is elliptic and {\it regular}, and  
will prove Thm. \ref{ellreg} of the introduction. 
 
Let $\wh G$ be the dual group of $G$, as in section \ref{dualtorsion}. We need the following consequence of \cite{reeder:thomae}: Let  $\wh g\in \wh G$ have  order $m$. Then 
\begin{equation}\label{thomae} \dim C_{\wh G}(\wh g)\geq \frac{|\check R|}{m},
\end{equation}
 with equality if and only if $\wh g$ is $\wh G$-conjugate to $\rho(e^{2\pi i/m})$.

\begin{prop}\label{ellregtrace}
Assume $w\in W$ is elliptic and regular of order $m$ and let $t\in T\cap w^G$.  Then $\Ad(t)$ has order $m$ and for all $y\in P/mQ$ we have $|R_t^+|\leq |\check R_y^+|$, with equality if and only if $y$ is $m$-principal.
\end{prop}
\proof It suffices to prove the result for one $t\in T\cap w^G$. 
Let $t=2\check\rho(\eta)$, where $\eta\in\bc^\times$ has order $2m$. 
Then $\Ad(t)$ has order $m$ and from 
 \cite[Prop.8]{rlyg} we have
$ t\in w^G$. 

Since $\Ad(t)$ has order $m$ it follows that $\Ad(g)$ has order $m$ for every $g\in w^G$.
Let $n\in w$. Since $w$ is elliptic regular and $\Ad(n)$ has order $m$, the group $\la n\ra$ freely permutes the root spaces of $T$ in $\fg$ and have no nonzero invariants in $\ft$. Therefore we have
\[\dim C_G(t)=\dim C_G(n)=\dfrac{|R|}{m}.\]

Let $\zeta\in S^1$ have order $m$. Given any $\nu\in P$, the element $\nu(\zeta)\in \wh T[m]$ depends only on the coset $y=\nu+mP\in P/mP$. As in \eqref{PQ}, we have a $W$-equivariant isomorphism 
\[P/mP\to \wh T[m], \qquad y\mapsto \hat y=\nu(\zeta),\quad\text{for any $\nu\in y$}.\]

Now take $y\in P/mQ$ and let $k$ be the order of $y+mP$ in the group $P/mP$. Clearly $k$ divides $m$. Applying  \eqref{thomae} gives the first inequality in the following. 
\begin{equation}\label{ineq}
\dim C_{\wh G}(\hat y)\geq \frac{|\check R|}{k}\geq \frac{|\check R|}{m}
=\frac{|R|}{m}=\dim C_G(t).
\end{equation}
Since $G$ and $\wh G$ have the same rank, we obtain the inequality 
$|R_t|\leq |\check R_y|$. 

If $y= v\rho+mQ$ for some $v\in W$, 
then applying \cite[Prop. 8]{rlyg} to $\wh G$ shows that 
$\hat y\in w^{\wh G}$. Since $w$ is elliptic regular of order $m$, we have
 $\dim C_{\wh G}(\hat y)=|\check R|/m$. 
 Now \eqref{ineq} implies that $|R_t|=|\check R_y|$. 

Conversely, if $|R_t|= |\check R_y|$ then both inequalities in \eqref{ineq} become equalities. Thus, $k=m$ and 
$\dim C_{\wh G}(\hat y)=|\check R|/m$. The condition for equality in \eqref{thomae} implies that $\hat y\in w^{\wh G}$.  It follows that $\hat y$ is $W$-conjugate to $\rho(\zeta)$, which means the image $y'$ of $y$ in $P/mP$ is $W$-conjugate to $ \rho+mP$. We may therefore assume $y'=\rho+mP$ and it remains 
to prove that $y$ is $W$-conjugate to $ \rho+mQ$. 

In the exact sequence \eqref{PQT}
\[1\lra\wh Z_{sc}\lra \pi^{-1}(\wh T[m])\overset{\pi}{\lra} \wh T[m]\lra 1, \]
the stabilizer $W_{\hat y}=\{v\in W:\ \hat y^v=\hat y\}$
acts on the fiber $\pi^{-1}(\hat y)$. Since $w$ is elliptic, this action of $W_{\hat y}$  on $\pi^{-1}(\hat y)$ is transitive, by Lemma \ref{transitive} applied to the group $\wh G_{sc}$.
Passing to the exact sequence \eqref{PQ}
\begin{equation}
0\lra P/Q\overset{m}{\lra} P/mQ\lra P/mP\lra 0,
\end{equation}
 this means that $W_{\hat y}$ acts transitively on the fiber in $P/mQ$
above $\rho+mP$. Since $y$ and $\rho+mQ$ belong to this fiber, the proof is complete. 
\qed

Combining Cor. \ref{principal}, Prop. \ref{elltrace} (a)  and Prop. \ref{ellregtrace} we have proved Theorem \ref{ellreg} in the introduction, restated here. 

\begin{thm}\label{thm:ellreg} Let $w\in W$ be elliptic and regular of order $m$ and let $y\in P/mQ$. 
\begin{enumerate}
\item If $y$ is not in the $W$-orbit of $\rho+mQ$ then 
for all  $\mu\in y\cap P_{++}$ we have 
\[\tr(w,V_\mu^T)=0.\]
\item If $vy=\rho+mQ$ for some $v\in W$ then for all 
$\mu\in y\cap P_{++}$ we have 
\begin{equation}\label{eq:ellreg}\tr(w,V_\mu^T)=\vep(v)\prod_{\check\al\in\check R_m^+}\dfrac{\la v\mu,\check\al\ra}{\la \rho,\check\al\ra}.
\end{equation}

\end{enumerate}
\end{thm}

{\bf Remark.\ } We have observed that the polynomial $\prod_{\check\al\in\check R_m^+}\check\al$ transforms by the sign character under the reflection subgroup $W_m=\la r_\al:\check\al\in\check R_m^+\ra$, so that 
the product 
\[\vep(v)\prod_{\check\al\in\check R_m^+}\la v\mu,\check\al\ra\]
is independent of the choice of $v\in W$ for which $v\mu\in \rho+mQ$. In fact there is a unique such $v$ for which $v^{-1}\check R_m^+\subset \check R^+$. 
With this choice,  \eqref{eq:ellreg} becomes
\begin{equation}\label{eq:ellreg2}\tr(w,V_\mu^T)=\vep(v)\prod_{\check\al\in\check R_\mu^+}\dfrac{\la \mu,\check\al\ra}{\la \rho,\check\al\ra}.
\end{equation}
where 
$\check R_\mu^+=\{\check\al\in\check R^+:\ \la \mu,\check\al\ra\in m\bz\}$.
Version \eqref{eq:ellreg2} has the computational advantage that each term in the product is positive. 

\section{A general character formula for zero weight spaces}\label{sec:general}

The general formula for $\tr(w,V_\mu^T)$ interpolates between the cases where $\dim T_w=0$ and $T_w=T$.

\subsection{The Weyl Character Formula and subtori}
\label{sec:wcf}
We return to the Weyl Character Formula, now to  
analyze its restriction to  cosets of subtori in $T$. At the moment there are no further specifications on these cosets. 

Fix an element $t\in T$ and a subtorus $S\subset T$. 
We recall more notation from the introduction: The set of positive roots $R^+$ is partitioned as 
\[R^+=R_{tS}^+\sqcup R^{tS}_1\sqcup R^{tS}_2,\]
where 
\[\begin{split}
R_{tS}^+&=\{\al\in R^+:\ e_\al\equiv 1\ \text{on $tS$}\}, \\
R^{tS}_1&=\{\al\in R^+:\ e_\al\equiv e_\al(t)\neq 1\ \text{on $tS$}\},\\ 
R^{tS}_2&=\{\al\in R^+:\ e_\al \ \text{is nonconstant on $tS$}\}.
\end{split}
\]
Also, we set
\[\rho_{tS}=\frac{1}{2}\sum\limits_{\al\in R_{tS}^+}\al,\quad
H_{tS}(\lam)=\prod\limits_{\al\in R_{tS}^+}\dfrac{\la \lam,\check\al\ra}{\la\rho_{tS},\check\al\ra},\quad 
W^{tS}=\{v\in W:\ v^{-1}R_{tS}^+\subset R^+\}
\]
\[\De=e_\rho\prod_{\al\in R^{tS}_1}(1-e_{-\al}),\qquad
S_0=\{s\in S:\ e_\al(ts)\neq 1\quad \forall \al\in R^{tS}_2\}.\]
Let $W_{tS}$ be the subgroup of $W$ generated by $\{r_\al:\ \al\in R_{tS}^+\}$. The product mapping 
$W_{tS}\times W^{tS}\to W$ is a bijection. 

\begin{lemma}\label{singularweyl} Let $tS$ be a coset of the subtorus $S\subset T$. For all $\mu\in P_{++}$ and $s\in S_0$ we have
\[\tr(ts,V_\mu)=
\sum_{v\in W^{tS}}
\frac{\vep(v)\cdot H_{tS}(v\mu)}
{\prod\limits_{\al\in R_1^{tS}}(1-e_{-\al}(t))}\cdot \frac{e_{v\mu-\rho}(ts)}
{\prod\limits_{\be\in R_2^{tS}}(1-e_{-\be}(ts))}.
\]
\end{lemma}

\proof Let $T'$ be the covering torus of $T$ with character group $\frac{1}{2}P$ (cf. \cite[VI.3.3]{bour456}). 

We again write  $e_\lam:T'\to S^1$ for the character of $T'$ corresponding to $\lam\in \frac{1}{2}P$. 

The Weyl group $W$ acts on $T'$ via its action on $\frac{1}{2}P$, and $W$ acts on the character ring $\bc[T']$ as $(w\cdot f)(t)=f(t^w)$. 
Let $A$ and $A_{tS}$ be the operators on $\bc[T']$ given by 
\[A(f)=\sum_{w\in W}\vep(w)w\cdot f, \qquad  \qquad
A_{tS}(f)=\sum_{w\in W_{tS}}\vep(w)w\cdot f.\]
Then 
\[A(f)=\sum_{v\in W^{tS}}\sum_{u\in W_{tS}}\vep(uv)(uv\cdot f)=
\sum_{v\in W^{tS}}\vep(v)A_{tS}(v\cdot f).\]
In particular for $\mu\in P$ we have 
\begin{equation}\label{A}
A(e_\mu)=\sum_{v\in W^{tS}}\vep(v)A_{tS}(e_{v\mu}).
\end{equation}

When restricted to $T$, the character $\chi_\mu$ of $V_\mu$  may be regarded as a function on $T'$, via the covering map $T'\to T$. It is given by the traditional expression of 
Weyl's Character Formula:
\begin{equation}\label{wcf1}
\chi_\mu=\frac{A(e_{\mu})}{A(e_\rho)}.
\end{equation}
On the right side of \eqref{wcf1} the numerator and denominator are functions on $T'$ whose zeros must be cancelled in order to evaluate $\chi_\mu$. 

From Weyl's identity (applied to both $W$ and $W_{tS}$), we have 
\begin{equation}\label{D}A(e_\rho)=D^{tS}\cdot D_{tS}=D^{tS}\cdot A_{tS}(e_{\rho_{tS}}),
\end{equation}
where
\[
D_{tS}=\prod_{\al\in R_{tS}^+}\left(e_{\al/2}-e_{-\al/2}\right),\qquad 
D^{tS}=\prod_{\al\in R_{1}^{tS}\sqcup R_2^{tS}}\left(e_{\al/2}-e_{-\al/2}\right),\qquad 
e_{\rho_{tS}}=\prod_{\al\in R_{tS}^+}e_{\al/2}
\] 
are all elements of $\bc[T']$. 

Recall $t\in T$ has been fixed. Now we also fix $s\in S_0$ and let $x\in T'$ be a lift of $ts$. 
Also let $z\in T'$ be arbitrary. For all $u\in W_{tS}$,  $v\in W^{tS}$ and $\mu\in \frac{1}{2}P$ we have
\[e_{uv\mu}(xz)=e_{v\mu}(x)\cdot e_{uv\mu}(z).
\]
It follows that 
\[A_{tS}(e_{v\mu})(xz)=e_{v\mu}(x)\cdot A_{tS}(e_{v\mu})(z).
\]
From \eqref{D} we have
\[A(e_\rho)(xz)=D^{tS}(xz)\cdot D_{tS}(xz).\]
For all $\al\in R_{tS}^+$ we have $e_\al(st)=1$, so 
$e_{\al/2}(x)=\pm 1$. It follows that   $e_{\rho_{tS}}(x)=\pm 1$ and that 
\[D_{tS}(xz)=e_{\rho_{tS}}(x)\cdot D_{tS}(z)=e_{\rho_{tS}}(x)A_{tS}(e_{\rho_{tS}})(z),\]
so 
\begin{equation}\label{wcf2}
A(e_\rho)(xz)=e_{\rho_{tS}}(x)\cdot D^{tS}(xz)\cdot A_{tS}(e_{\rho_{tS}})(z).
\end{equation}

So far $x$ is fixed but we have made no restriction on $z$. Since the set of elements in $T'$ on which no root $=1$ is dense in $T'$, we can choose a sequence $z_n\to 1$ in $T'$ such that $e_\al(xz_n)\neq 1$ for all $n$ and all $\al\in R$. 

From \eqref{wcf1} and \eqref{wcf2} we have that 
\[\chi_\mu(xz_n)=\frac{1}{e_{\rho_{tS}}(x)\cdot D^{tS}(xz_n)}
\sum_{v\in W^{tS}}\vep(v)e_{v\mu}(x) \frac{A_{tS}(e_{v\mu})}{A_{tS}(e_{\rho_{tS}})}(z_n).
\]
 Letting $n\to\infty$ we get, from the Weyl dimension formula applied to the (connected) centralizer  $C_L(t)$, where $L=C_G(S)$,  that 
\[\chi_\mu(x)=\frac{1}{e_{\rho_{tS}}(x)\cdot D^{tS}(x)}
\sum_{v\in W^{tS}}\vep(v) H_{tS}(v\mu) e_{v\mu}(x).
\]
Since
\[e_{\rho_{tS}}\cdot D^{tS}=e_{\rho}\cdot 
\prod_{\al\in R_1^{tS}\sqcup R_2^{tS}}(1-e_{-\al}),\]
it follows that 
\[\chi_\mu(x)=\sum\limits_{v\in W^{tS}}
\frac
{\vep(v)H_{tS}(v\mu)}{ \prod\limits_{\al\in R_1^{tS}\sqcup R_2^{tS}}(1-e_{-\al}(x))} e_{v\mu-\rho}(x).
\]
Since $\mu$ and all $v\mu-\rho$ lie in $P$ we may replace $x$ by its projection $st\in T$.
Decomposing the product according to $R_1^{tS}\sqcup R_2^{tS}$
we obtain the formula in Lemma \ref{singularweyl} . 
\qed

\subsection{Characters and constant terms}
\label{sec:charconst}

Given $w\in W$, let $S=(T_w)^\circ$ be the identity component of the fixed-point group $T_w:=\{t\in T: t^w=t\}$, and let $Y$ be the character lattice of $S$. 

The rank of $Y$ equals the rank of the character lattice $P_w=\{\lam\in P:\ w\lam=\lam\}$ of the quotient torus $T/[w,T]$, where $[T,w]=\{t^w\cdot t^{-1}:\ t\in T\}$.  

For $s\in S\cap [T,w]$  one checks that $s^m=1$, where $m$ is the order of $w$. This implies that the composite mapping
\begin{equation}\label{isogeny}
S\hra T\lra T/[T,w]
\end{equation}
has finite kernel, hence is surjective.  It follows that the restriction map $P\to Y$ is injective on $P_w$.  

Recall that we regard $w$ as a subset of $N$.
Fix $n\in w$. 
From the surjectivity of \eqref{isogeny} it follows that every element of $w$ is $T$-conjugate to an element of $nS$. Indeed, given any $t\in T$, there exist $s\in S$ and $z\in T$ such that $s=(z^w\cdot z^{-1})\cdot t$, so we have 
\[z(nt)z^{-1}=n(z^w\cdot z^{-1})t=ns.\]

Let $\bc[S]$ be the character ring of $S$. Every element $\vp\in \bc[S]$ uniquely expressed as a
finite linear combination
\[\vp=\sum_{\nu\in Y} \hat\vp(\nu) e_\nu\]
with coefficients $\hat\vp(\nu)\in \bc$. The {\it constant-term functional} $\int_S:\bc[S]\to\bc$ is given by 
\[\int_S \vp(s)\ ds=\hat\vp(0).\]

\begin{lemma}\label{constant} Let $V$ be a finite-dimensional representation of $G$ and let $w\in W$. Then the trace of $w$ on the zero weight space $V^T$ is given by
\[\tr(w,V^T)=\int\limits_S\tr(ns,V)\ ds,\]
for any choice of $n\in w$. 
\end{lemma}
\proof For $\lam\in P$, let $V^\lam=\{v\in V:\ tv=e_\lam(t)v\ \forall t\in T\}$.  We have $nV^\lam=V^{w\lam}$. It follows that for all $s\in S$ we have
\[\tr(ns,V)=\sum_{\lam\in P_w}\tr(ns, V^\lam)=
\sum_{\lam\in P_w}\tr(n, V^\lam)\cdot e_\lam(s).
\]
Since  the restriction map $P\to Y$ is injective on $P_w$
it follows that 
\[
\int\limits_S\tr(ns,V)\ ds=\tr(n,V^T)=\tr(w,V^T).\]
\qed


Since $n\in C_G(S)$, there exists $g\in C_G(S)$ such that $n^g\in T$. We set 
\[t:=n^g\in T.\]
Since $g$ centralizes $S$, we have $(ns)^g=ts$ for all $s\in S$. It follows that  
\[\tr(ns,V)=\tr(ts,V),\]
as functions on $S$.
From Lemma \ref{constant} we have 

\begin{prop}\label{prop:constant} Let $w\in W$ and let $S=(T_w)^\circ$ be the connected fixed-point subtorus of $w$ in $T$. Then there exists 
$t\in T$ which is $C_G(S)$-conjugate to an element of $w$ and for any such $t$ 
the trace of $w$ on the zero weight space $V^T$ is given by
\[\tr(w,V^T)=\int\limits_S\tr(ts,V)\ ds.\]
\end{prop}

\subsection{Formal Characters}
In this section we extend the functional $\int_S$ to certain rational functions on $S$, using the theory of formal characters \cite[22.5]{humphreys}. 

Let $S$ be a compact (or algebraic) torus with character lattice $Y$ and coordinate algebra $\bc[S]$. 
The Fourier expansion
\[\vp=\sum_{\nu\in Y}\hat\vp (\nu)e_\nu, \qquad \text{for $\vp\in\bc[S]$}\]
defines a $\bc$-algebra isomorphism (Fourier transform) $\vp\mapsto\hat \vp$ from $\bc[S]$ to the group algebra $\bc[Y]$, where $\bc[S]$ has the pointwise product and $\bc[Y]$ has the convolution product: 
\begin{equation}\label{convolution}
\phi\ast \psi(\nu)=\sum_{\lam+\mu=\nu}\phi(\lam)\psi(\mu). 
\end{equation}
Via this isomorphism we identify the constant term functional $\int_S:\bc[S]\to\bc$ with the evaluation functional $\phi\mapsto \phi(0):\bc[Y]\to\bc$. 

Let us be given nonzero elements $\nu_1,\dots,\nu_r\in Y$, not necessarily distinct. Let $\Ga\subset Y$ be the semigroup generated by $\{\nu_1,\dots,\nu_r\}$. 

We let $\bc((Y))$ be the set of functions $\phi:Y\to\bc$ supported on finitely many sets of the form $\lam-\Ga$ for some $\lam\in Y$. 
The convolution product \eqref{convolution} extends to $\bc((Y))$ and $\bc[Y]$ is a subalgebra of $\bc((Y))$. 
We extend the functional $\int_S$ to $\bc((Y))$ by setting
\[\int_S\phi =\phi(0).\]
For any complex number $z$ and $i=1,\dots, r$ the 
Fourier transform of $1-z e_{-\nu_i}$ becomes invertible in $\bc((Y))$; its inverse is the function whose support is contained in $-\bn\nu_i$ and whose value at $-n\nu_i$, for $n\in\bn$, is $z^n$. More generally, for
\begin{equation}\label{h}
h=\prod_{i=1}^r(1-z_ie_{-\nu_i})\in \bc[S]
\end{equation}
the convolution-inverse of
$\hat h$ 
given by the weighted partition function
\begin{equation}\label{Pmu}
\hat h^{-1}(-\nu)=\scp(\nu)=\sum z_1^{n_1}z_2^{n_2}\cdots z_r^{n_r}
\end{equation}
where the sum runs over all $(n_1,\dots, n_r)\in\bn^r$ such that 
$\sum\limits_{i=1}^rn_i\nu_i=\nu$. Note that $\hat h^{-1}$ is supported on $-\Ga$ and that $(\hat e_\nu\ast \hat h^{-1})(0)=\scp(\nu)$.

\begin{lemma}\label{formalfourier} Let $h\in\bc[S]$ have the form \eqref{h} and suppose $f,g\in \bc[S]$ are such that $h\cdot f=g$. 
Then 
\[\int_Sf(s)\ ds=
\sum_{\nu\in Y}\hat g(\nu) \scp(\nu),\]
where $\scp(\nu)$ is given by \eqref{Pmu}.
\end{lemma}
\proof We have $\hat g=\widehat{f\cdot h}=\hat f\ast\hat h$ in $\bc[Y]$ hence also in $\bc((Y))$. Since $\hat h$ is invertible in $\bc((Y))$ it follows that 
\[\hat f=\hat g\ast\hat h^{-1}=\sum_{\nu\in Y}\hat g(\nu)\hat e_\nu\ast\hat h^{-1}, \]
so we have 
\[
\int_Sf(s)\ ds=\hat f(0)=\sum_\nu \hat g(\nu)(\hat e_\nu\ast\hat h^{-1})(0)=\sum_\nu \hat g(\nu) P(\nu).
\]

\qed

\subsection{The character formula} \label{sec:generalw}

We apply the results (and notation) of Section \ref{sec:wcf} to the situation of \ref{sec:charconst}.  

Enumerate $R^{tS}_2=\{\be_1,\dots, \be_r\}$, let $\nu_i\in Y$ be the restriction of $\be_i$ to $S$, and let $z_i=e_{-\be_i}(t)$. 
As in \eqref{Pmu}, for $\nu\in Y$ we set
\[\scp_w(\nu)=\sum z_1^{n_1}z_2^{n_2}\cdots z_r^{n_r}
\]
where the sum runs over all $(n_1,\dots, n_r)\in\bn^r$ such that 
$\sum\limits_{i=1}^rn_i\nu_i=\nu$. If there are no such $r$-tuples $(n_i)$ then $\scp_w(\nu)=0$. 

\begin{thm}\label{mainthm} The trace of $w\in W$ on the zero weight space $V_\mu^T$ is given by
\begin{equation}\label{eq:mainthm}
\tr(w, V_\mu^T)=\frac{1}{\prod\limits_{\al\in R_1^{tS}}(1-e_{-\al}(t))}
\sum_{v\in W^{tS}}\vep(v)H_{tS}(v\mu)e_{v\mu-\rho}(t)\scp_w(v\mu-\rho), 
\end{equation}
where $S=(T_w)^\circ$ and
$t\in T$ is  $C_G(S)$-conjugate to an element of $w$ 
(see Proposition \ref{prop:constant}).
\end{thm}

\proof From Proposition \ref{prop:constant} we have  
\[\tr(w,V_\mu^T)=\int_S\tr(ts, V_\mu)\ ds.
\]
From Lemma \ref{singularweyl} we have on $S_0$ the relation
\[\tr(ts,V_\mu)=
\sum_{v\in W^{tS}}
\frac{\vep(v)\cdot H_{tS}(v\mu)}
{\prod\limits_{\al\in R_1^{tS}}(1-e_{-\al}(t))}\cdot \frac{e_{v\mu-\rho}(ts)}
{\prod\limits_{\be\in R_2^{tS}}(1-e_{-\be}(ts))}.
\]
For $v\in W^{tS}$ we set
\[c_v=\vep(v)\cdot \frac{ H_{tS}(v\mu)e_{v\mu-\rho}(t)}{\prod\limits_{\al\in R_1^{tS}}(1-e_{-\al}(t))}.
\]
In Prop. \ref{formalfourier} we take $f,g,h\in \bc[S]$ defined by 
\[f(s)=\tr(ts,V_\mu),
\qquad
g(s)=\sum_{v\in W^{tS}}
c_v e_{v\mu-\rho}(s),\qquad
h(s)=\prod\limits_{\be\in R_2^{tS}}(1-e_{-\be}(ts)).
\]

Since $f\cdot h=g$,  we have from Lemma \ref{formalfourier} that 
\[\begin{split}
\tr(w,V_\mu^T)&=\int_S\tr(ts,V_\mu)\ ds=\hat f(0) \quad (\text{ Prop. \ref{prop:constant}})\\
&=\hat g\ast \hat h^{-1}(0)
=\sum_{v\in W^{tS}}c_v\hat e_{v\mu-\rho}\ast \hat h^{-1}(0)
=\sum_{v\in W^{tS}}c_v\hat \scp_w(v\mu-\rho).
\end{split}
\]
Theorem \ref{general} is proved. \qed

\subsubsection{A special case}

Let $J$ be a subset of the simple roots in $R^+$, let
$R_J^+$ be the positive roots spanned by $J$ and $W_J=\la r_\al:\ \al\in R_J^+\ra$. Suppose $w$ is a Coxeter element in $W_J$.
Then $S=(\cap_{\al\in J} \ker e_\al)^\circ$. 
Let $G_J$ be the derived group of $C_G(S)$. We may choose $n\in w\cap G_J$, and let $t$ be a $G_J$-conjugate of $n$. Then $t$ is a regular element of $G_J$, so we have
\begin{equation}\label{coxeterinlevi1}
R_{tS}^+=\varnothing,\qquad R_1^{tS}=R_J^+,\qquad 
R_2^{tS}=R^+-R_J^+,\qquad H_{tS}\equiv 1.
\end{equation}
In this situation  \eqref{eq:mainthm}  becomes
\begin{equation}\label{coxeterinlevi2}
\tr(w, V_\mu^T)=\frac{1}{\prod\limits_{\al\in R_J^+}(1-e_{-\al}(t))}
\sum_{v\in W}\vep(v)e_{v\mu-\rho}(t)\scp_w(v\mu-\rho). 
\end{equation}

\section{Small Groups}\label{sec:smallgroups}

In this section we work out the complete character formula for $\tr(w,V_\mu^T)$ in the cases where Kostant's partition function has a closed quasi-polynomial expression (that I am aware of).  
This means $G$ is one of $\SU_3,\ \Sp_4,\ G_2,\ \SU_4$.

We use the following notation.  
For $a\in\bz$, let $(a)_2=0$ if $a$ is even, $(a)_2=1$ if $a$ is odd.
For $n\in\{3,4,6\}$ let 
\[(a)_n=\begin{cases}
\ \ \ 1&\quad\text{if}\quad a\equiv 1\mod n\\
-1&\quad\text{if}\quad a\equiv -1\mod n\\
\ \ \ 0&\quad\text{if}\quad \gcd(a,n)>1.
\end{cases}
\]

\subsection{Rank-Two Groups}

For the rank-two groups, let $\al$ and $\be$ be simple roots, where $\al$ is a short root for  $\Sp_4$ and is a long root for $G_2$. 
$\al_0$ is the highest root and $\be_0$ is the highest short root.
Let $\om_\al$ and $\om_\be$ be the corresponding fundamental weights, so that $\rho=\om_\al+\om_\be$. 
For $\mu\in P_{++}$, we have  
\begin{equation}\label{rho-shift}
\mu=a\om_\al+b\om_\be
\end{equation}
where $a,b$ are positive integers. The character of $W$ afforded by $V_\mu^T$ will be expressed in terms of quasi-polynomial functions of $(a,b)$.

Recall that $\bn=\{0,1,2,3,\dots\}$.

\subsubsection{$\SU_3$}\label{SU3} Here $V_\mu^T\neq 0$ if and only if $a\equiv b\mod 3$. Since 
$V$ and its dual have isomorphic zero weight spaces, we may assume that $a\leq b$. 
The conjugacy classes in $W=S_3$ are represented by $1=1_W$, a reflection $r$, and a Coxeter element $\cox$.  
We find 
\[\begin{split}
\dim V_\mu^T&=a\\
\tr(r,V_\mu^T)&=(a)_2\cdot (-1)^{b+1}\\
\tr(\cox,V_\mu^T)&=(a)_3
\end{split}
\]
The computations are summarized as follows. 
Kostant's partition function $\scp$ is given by $\scp(m\al+n\be)=1+\min(n,m)$, for $n,m\in\bn$. From this one easily computes the above value for $\dim V_\mu^T$. 

The value for $\tr(\cox,V_\mu^T)$ is obtained from Kostant's formula. Setting $\zeta=\eta^2$, we get
\[\tr(\cox,V_\mu^T)=\tr(\check\rho(\zeta),V_\mu)
=\frac{ \zeta^a-\zeta^{-a} }{\zeta-\zeta^{-1}}=(a)_3.\]


For $r=r_\al$, we have 
$S=\check\lam(S^1)$, where $\check\lam(z)=\diag(z,z,z^{-2})$, and we may take $t=\check\rho(-1)=\diag(-1,1,-1)$. We have (cf. \ref{coxeterinlevi1}) 
\[R_{tS}^+=\varnothing,\quad R_1^{tS}=\{\al\},\quad R_2^{tS}=\{\be,\rho\}.\]
We have $Y=\bz\nu$, where $e_\nu(\check\lam(z))=z$, $z_1=\be(t)=-1$, $z_2=\rho(t)=1$, $\nu_1=\nu_2=3\nu$ and 
\[\scp_r(n\nu)=\begin{cases}
1-(n/3)_2&\quad \text{if $n\in 3\bn$}\\
0&\quad \text{if $n\notin 3\bn$}
\end{cases}
\]
From \ref{coxeterinlevi2} we get
\[\begin{split}
\tr(r,V_\mu^T)&=\frac{1}{2}\sum_{v\in W}\vep(v)e_{v\mu}(t)\scp_r(v\mu-\rho)\\
&=\frac{1}{2}\left\{(-1)^{a+b}\left[1-\left(\frac{a+2b-3}{3}\right)_{\!2}\ \right]
-(-1)^{b}\left[1-\left(\frac{a+2b-3}{3}\right)_{\!2}\ \right]\right\}\\
&=(-1)^{b+1}(a)_2\left(\frac{a+2b}{3}\right)_{\!2}=(-1)^{b+1}(a)_2.
\end{split}
\]
Using the representation theory of $\SU_2$, one can also obtain $\tr(r, V_\mu^T)$ by calculating multiplicities of weights $k\al$ for $k\in\bn$. 
For $\SU_3$ this is feasible because the partition function $\scp$ is so simple.

{\bf Remarks:\ } 

1)\ $V_\mu^T$ is a multiple of the regular representation $\Reg$ of $W$ if and only if  $a\in 6\bz$, in which case $V_\mu^T$ is $a/6$ copies of $\Reg$.

2)\  $V_\mu^T$ is irreducible if and only if $a\in{1,2}$. 
If $a=1$, then $V_\mu=\Sym^{3d}(\bc^3)^\ast$ where $b=1+3d$, and  we have $V_\mu^T\simeq\vep^d$. 
If $a=2$ then $V_\mu^T\simeq \varrho$. 
Thus, each irrep of $W$ is of the form $V_\mu^T$ for infinitely many $\mu\in P_{++}$.

\subsubsection{$\Sp_4$}\label{Sp4}

In this section we explicate our character formula for $G=\Sp_4$. 
The simple roots are  $\al,\be$ where $\al$ is short. 
Writing $\mu=a\om_\al+b\om_\be$, we have 
$V^T_\mu\neq 0$ iff $a$ is odd. For $\eta\in S^1$ of order eight, let 
\[(a)_8=\frac{\eta^a-\eta^{-a}}{\eta-\eta^{-1}}=
\begin{cases} 
\ \ 1&\quad \text{if $a\equiv 1,3\mod 8$}\\
-1&\quad \text{if $a\equiv 5,7\mod 8$}.
\end{cases}
\]
The conjugacy classes in $W$ are represented by 
\[1_W,\ r_\al,\  r_\be,\ \cox^2,\ \cox.\]
For $\mu=a\om_\al+b\om_\be$ the values of $\tr(w,V_\mu^T)$  are shown in the following table. 
\begin{center}
\[
\renewcommand{\arraystretch}{2.6}
\begin{array}{l|l}
w&\tr(w,V_\mu^T)\\
\hline
1_W&
\dfrac{1}{2}\left(ab+(b)_2\right)\\
r_\al & \dfrac{(-1)^{b+1}}{2}\cdot
\left[b+(a)_4(b)_2\right]\\
r_\be & 
\begin{cases}
(a)_4&\quad\text{if $b\in 2+4\bn$}\\
(b)_4&\quad\text{if $a+b\in 2+4\bn$}\\
\ \ 0&\quad\text{otherwise}
\end{cases}\\
\cox^2&\dfrac{(a)_4}{2}\left[b+a(b)_2\right] \\
\cox&\begin{cases}
-(a)_8(a)_4&\quad\text{if $b\in 2+4\bn$}\\
\ \ \ (a)_8(b)_4&\quad\text{if $a+b\in 2+4\bn$}\\
\quad 0&\quad\text{otherwise}
\end{cases}\end{array}
\]
\end{center}

The calculations are summarized as follows. 

The formula for $\dim V_\mu^T$ is an exercise in \cite{bour76}. 
The traces of $\cox$ and $\cox^2$ come from Theorem \ref{thm:ellreg}.

For $w=r_\al$ we have $S=\check\be_0(S^1)$, where $\check \be_0$ is the highest  coroot. Then $Y=\bz\nu$ with $\nu(\check\be_0(z))=z$,  and we can take $t=\check\al(i)$, where $i^2=-1$. Then
\[R_{tS}^+=\varnothing,\qquad R^{tS}_1=\{\al\},\qquad R_2^{tS}=\{\be, \al+\be, 2\al+\be\}.\]
Each root in $R_2^{tS}$ restricts to $2\nu$ on $S$. 
The partition function is given by 
\[\scp_{r_\al}(n\nu)=\sum_{(n_1,n_2,n_3)}(-1)^{n_1+n_3},\]
summed over $\{(n_1,n_2,n_3)\in \bn^3:\ 2n_1+2n_2+2n_3=n\}$, so 
\[\scp_{r_\al}(n\nu)=
\begin{cases}
(-1)^{n/2}\left ( 1+\lfloor n/4 \rfloor \right)&\quad\text{if $n\in 2\bn$}\\
0&\quad\text{if $n\notin 2\bn$}.
\end{cases}
\]
In the sum 
\[
\tr(r_\al,V_\mu^T)=
\frac{1}{\De(t)}\sum_{v\in W}\vep(v)e_{v\mu}(t)\scp_{r_\al}(v\mu-\rho),
\]
only the terms for $v=1, r_\al, r_\be, r_\al r_\be$ are nonzero and we get 
\[
\tr(r_\al,V_\mu^T)=(a)_4\left[ \scp_{r_\al}(a+2b-3)-(-1)^b\scp(a-3)\right],
\]
which simplifies to the formula above. The trace of $r_\be$ is obtained similarly.

{\bf Remarks.\ } 

1. The irreps of $\Sp_4$ with nonzero weight spaces are those factoring through $\SO_5$. We see that the long element $w_0=\cox^2$ has nonzero trace on $V^T$ for  every irrep $V$ of $\SO_5$. This means the virtual character $V^T$ is never a linear combination of induced representations from proper parabolic subgroups of $W$ (and in particular is never the regular representation). 

In fact, this holds for any irrep $V$ of $\SO_{2n+1}$, $n\geq 1$. Indeed, $P/2Q=A\sqcup (\rho+A)$, where $A=Q/2Q$. The $W$-stabilizer of $\rho$ in $\rho+A$ is the symmetric group $S_n$, so  $W$ is transitive on $\rho+A$. Hence $\tr(w_0,V^T)\neq 0$ by Thm. \ref{thm:ellreg}.

2. Back to $\Sp_4$. Comparing $\dim V_\mu^T$ and $\tr(w_0,V_\mu^T)$ shows that $w_0$ is a scalar on $V_\mu^T$ if and only if $a=1$ or $b=1$. In the former case $V_\mu$ is the degree $b-1$ harmonic polynomials on the five-dimensional orthogonal representation and $w_0$ is trivial on $V_\mu^T$. In the latter case $V_\mu$ is the degree $a-1$ polynomials on the four-dimensional symplectic representation and $w_0$ acts on $V_\mu^T$ by the scalar $(a)_4$.

3. The trivial representation of $W$ appears in $V_\mu^T$ with multiplicity 
\[\la \triv, V_\mu^T\ra=\begin{cases}
\dfrac{b}{4}\left\lfloor \dfrac{a}{4}\right\rfloor\quad&\text{if $b\in 4\bn$}\\
&\\
\dfrac{1}{4}\left(b\left\lfloor \dfrac{a}{4}\right\rfloor+(a)_4(1-(a)_8)\right)
\quad&\text{if $b\in 2+4\bn$}\\
&\\
\dfrac{1}{4}(b+(a)_4)\left\lceil \dfrac{a}{4}\right\rceil\quad&\text{if $a+b\in 4\bn$}\\
&\\
\dfrac{1}{4}\left( (b+(a)_4)\left\lceil \dfrac{a}{4}\right\rceil+(b_4)(1+(a)_8)\right)\quad&\text{if $a+b\in 2+4\bn$}
\end{cases}
\]
The reflection representation appears with multiplicity
\[\la \refl, V_\mu^T\ra=\frac{1}{8}[a-(a)_4][b-(b)_2(a)_4].\]

\newpage

\subsubsection{ $G_2$} \label{G2}

In this section we explicate our character formula for $G$ of type $G_2$.

For relatively prime positive integers $m,n$, and $q\in\bq$, let
$P_{mn}(q)$ be the number of solutions $(x,y)\in\bn\times\bn$ of the equation $mx+ny=q$. For any integer $k$, we have the constant-term formula
\[
\int_{S^1}\frac{z^k}{(z^{m}-z^{-m})(z^{n}-z^{-n})}\ dz=P_{mn}\left(\frac{k-m-n}{2}\right).
\]
Now $P_{mn}(q)=0$ unless $q\in\bn$, in which case $P_{mn}(q)$ is given by {\it  Popoviciu's formula} \cite{popoviciu} (see also \cite[1.3]{beck-robins})
\[P_{mn}(q)=\frac{q}{mn}-\left\{\frac{qm'}{n}\right\}-\left\{\frac{qn'}{m}\right\}+1,\]
where $m', n'$ are any integers such that $mm'+nn'=1$, and 
$\{r\}=r-\lfloor{r}\rfloor$ denotes the fractional part of a rational number $r$.

For $G_2$ we will need just two of these partition functions, $P_{12}$ and $P_{23}$, which have the more explicit formulas
\[P_{12}(q)=1+\gint{q}{2},\qquad P_{23}=1+\gint{q}{2}-\lint{q}{3}.\]

The conjugacy classes in $W$ are represented by 
\[1_W,\ r_\al,\  r_\be,\ \cox^3,\ \cox^2,\ \cox.\]
For $\mu=a\om_\al+b\om_\be$ the values of $\tr(w,V_\mu^T)$  are shown in the following table. 
\begin{center}
\[
\renewcommand{\arraystretch}{2.6}
\begin{array}{c|c}
w&\tr(w,V_\mu^T)\\
\hline
e&
\dfrac{a^3b}{12}+\dfrac{a^2b^2}{8}+\dfrac{ab^3}{36}+\dfrac{ab}{6}+(b)_3\cdot \dfrac{2a}{9}+(a)_2\cdot(b)_2\cdot \dfrac{3}{8}\\
r_\al &(-1)^{a+b}\left[
P_{23}\left(\dfrac{3a+2b-5}{2}\right)+
(-1)^a\cdot P_{23}\left(\dfrac{3a+b-5}{2}\right)-
P_{23}\left(\dfrac{b-5}{2}\right)\right]\\
r_\be &(-1)^{a+1}\left[
 P_{12}\left(\dfrac{2a+b-3}{2}\right)- P_{12}\left(\dfrac{a+b-3}{2}\right)+(-1)^b\cdot P_{12}\left(\dfrac{a-3}{2}\right)\right]\\
\cox^3&\dfrac{1}{8}\left[ (a)_2\cdot a(3a+2b)-(a+b)_2\cdot (a+b)(3a+b)+(b)_2\cdot (2a+b)b\right]\\
&=\dfrac{1}{8}\begin{cases}
0&(a,b)\equiv(0,0)\mod 2\\
-a(3a+2b)& (a,b)\equiv(0,1)\mod 2\\
-b(2a+b)&  (a,b)\equiv(1,0)\mod 2\\
(a+b)(3a+b)&(a,b)\equiv(1,1)\mod 2
\end{cases}
\\
\cox^2 &\dfrac{1}{9}\left[(a)_3\cdot (3a+2b)-(a+b)_3\cdot (3a+b)+(2a+b)_3\cdot b\right]\\
&=\dfrac{(b)_3}{3}\begin{cases}
-a&\quad \text{$ab\equiv 0\mod 3$}\\
2a+b&\quad \text{$ab\equiv 1\mod 3$}\\
-(a+b)&\quad \text{$ab\equiv 2\mod 3$}
\end{cases}
\\
\cox&-(a)_6\cdot (3a+2b)_6+(a+b)_6\cdot (3a+b)_6-(2a+b)_6\cdot (b)_6\\
&=\begin{cases}
+1\quad & (a,b)\equiv (\pm 1,\pm 1),\  (\pm 2, \pm 1),\ (3,\pm 2)\mod 6\\
-1\quad & (a,b)\equiv (\pm 1,\pm 2),\  (\pm 2, \mp 1),\ (3,\pm 1)\mod 6\\
0\quad & \text{otherwise}
\end{cases}
\end{array}
\]
\end{center}

As far as I know, the dimension of $V_\mu^T$ for $G_2$ was first given explicitly in \cite{aik}. The formulas therein evidently take positive integer values without any sign cancellations.  On the other hand, they have nine different cases, according to congruences of $(a,b)\mod(2,6)$.  
A more compact expression for $\dim V_\mu^T$ was obtained by Vergne, using coordinates $(m,n)$ where $\mu-\rho=m\al+n\be$ (see \cite{kumar-prasad}).  If we change coordinates to $(a,b)$ as in \eqref{rho-shift}
then Vergne's formula for $\dim V_\mu^T$ simplifies to the one given above. 

The character values on the reflections come from Theorem \ref{mainthm}. 

The elliptic character values are expressed as harmonic quasi-polynomials as in \eqref{intro:elliptictrace} and also as piecewise monomials determined by the action of $W$ on $P/mP$, as in Thm. \ref{ellreg}. 

The latter is analyzed by recalling that $W$ is the isometry group of the quadratic form $q:P\to\bz$ given  
by $q(\mu)=3a^2+3ab+b^2$, for which $q(\rho)=7$. 
For each $m\in\{2,3,6\}$, reduction modulo $m$ gives a quadratic form
$q_m:P/mP\to\bz/m\bz$  and a homomorphism 
from $W$ to the isometry group $O(q_m)$ of $q_m$ on $P/mP$.
For all $m$, we have $\tr(\cox^{6/m},V_\mu^T)\neq 0$ only if $q_m(\mu)=q_m(\rho)=1$.

The form $q_2$ is anisotropic; $W$ is transitive on nonzero vectors in $P/2P$ and $\tr(\cox^{3},V_\mu^T)\neq 0$ if and only if $q_2(\mu)=1$. 

The form $q_3$ is degenerate, with radical spanned by 
$\om_\al+3P$. 
The map $W\to O(q_3)$ realizes $W$ as the Borel subgroup of $\GL(P/3P)$ stabilizing the radical line.
The orbit $W\rho$ consists of the six vectors in $P/3P$ with $q_3(\mu)=1$; these are the cosets of the short roots. 

For $m=6$ the vectors with $q_6=1$ form two $W$-orbits, 
represented by $\rho+6P$ and $\be+6P$. It follows that 
$\tr(\cox,V_\mu^T)\neq 0$ if and only if $q_6(\mu)=1$ and $\mu$ is not in the coset of a root.

{\bf Remarks.\ }

1)\ $V_\mu^T$  is a multiple of the regular representation of $W$ if and only if $(a,b)\equiv (0,0)\mod (2,6)$. 

2)\ Let $W_2\simeq W_{A_1\times A_1}$ be a subgroup of $W$ generated by reflections about a pair of orthogonal roots. Then 
$V_\mu^T$ is induced from a character of $W_2$ if and only if $b\in3\bz$.

3)\ $V_\mu^T$ is irreducible for $W$ only for the trivial, seven-dimensional and adjoint representations.  Indeed, from the monomial formula for $\tr(\cox^3,V_\mu^T)$ we find that the list of $\mu\in P_{++}$ for which 
$\tr(w_0, V_\mu^T)$ belongs to the set $\{\chi(\cox^3):\ \chi\in\Irr(W)\}=\{\pm 1, \pm 2\}$, is just
\[(a,b)=(1,1),\quad (1,2),\quad (2,1),\quad (3,2).\]
The first three cases are those with $V_\mu^T$ irreducible. In the  last case $\mu=3\om_\al+2\om_\be$, we  have $\tr(\cox^3,V_\mu^T)=-2$ and $\tr(\cox^2, V_\mu^T)=+1$. As these are not the values of any $\chi\in\Irr(W)$, it follows that $V_\mu^T$ cannot be irreducible for $W$. 

\subsection{$\SU_4$} \label{SU4}

For $G=\SU_4$, let $\al,\be,\ga$ be simple roots with $\al$ and $\ga$ orthogonal to each other. Let $\om_\al, \om_\be, \om_\ga$ be the corresponding fundamental weights, so that $\rho=\om_\al+\om_\be+\om_\ga$. 

For a dominant regular weight $\mu$, the  character of $V_\mu^T$ is expressed as a function of coordinates $(a,b,c)$ where 
$a,b,c$ are positive integers such that
\[\mu=a\om_\al+b\om_\be+c\om_\ga.\]
Replacing $V_\mu$ by its dual does not change the $W$-module structure of $V_\mu^T$, so we may assume that 
\[a\leq c.\] 
Set
\[d:=\frac{c-a}{2},\qquad f:=b-d, \qquad s:=\dfrac{c+a}{2}.\]
Then $V_\mu^T\neq 0$ if and only if $f\in 1+2\bz$.

Index the conjugacy classes in $W=S_4 $ by partitions $[\lam_1\lam_2\cdots]$ of $4$. For example $[1111]$ is the identity element and $[4]=\cox$. 

The character values are shown in the following table. 

\begin{center}
\renewcommand{\arraystretch}{2.3}
\begin{tabular}{l | l}
$w$&$\tr(w,V_\mu^T)$\\
\hline
$[1111]$&
	$\dfrac{1}{2}\begin{cases}
	ab(a+b)&\quad f\leq 1\\
	a\left(bc+1-d^2\right)&\quad f> 1
	\end{cases}$\\

$[211]$&
         $\dfrac{1}{2}\begin{cases}
         a\cdot[d]_{2\bn}-b\cdot [s]_{2\bn}&\quad
	f\leq -1\\
	\left( 1+s(f)_4\right)\cdot [d]_{2\bn}-
	\left( 1+d(f)_4\right)\cdot [s]_{2\bn}&\quad
	f> -1\\
	\end{cases}$\\


$[22]$&
       $\dfrac{1}{2}\left\{ (-1)^db\cdot [c]_{1+{2\bn}}-(-1)^ca\cdot [d]_{2\bn}\right\}$\\
$[31]$&$(a)_3\cdot
\big(
[2c+f]_{3\bn}-[f]_{3\bn}\big)$\\
$[4]$&$\dfrac{1}{2}
\left\{
(-1)^a\de\left(\dfrac{s}{2}\right)+
(-1)^{b}\de\left(\dfrac{b+s}{2}\right)-
(-1)^c\de\left(\dfrac{d}{2}\right)
\right\}$
\end{tabular}
\end{center}

where for any $A\subset\bz$ and $x\in\bq$ we use the notation
\[
[x]_A=\begin{cases}
1&\quad\text{if $x\in A$}\\
0&\quad\text{if $x\not\in A$}
\end{cases}
\qquad
\de(x)=\begin{cases}
(-1)^x&\quad\text{if $x\in \bn$}\\
0&\quad\text{if $x\not\in \bn$}.
\end{cases}
\]


The dimension of $V_\mu^T$ was obtained in \cite[Thm. 6.1]{kumar-prasad}, using coordinates different from ours.  Our standing condition $a\leq c$ puts us in  case $(2)$ (if $f\leq 1$) or  $(4)$ (if $f>1$) of [loc. cit.]. 

We give here just the computation for the trace of a reflection on $V_\mu^T$, leaving the other (easier) cases to the reader. 

We take $w=r_\be$, $S=\{\diag(x,y,y,z):\ xy^2z=1\}$, 
$t=\diag(-1,-1,1,1)=\check\al(-1)$.
Then 
$R_{tS}^+=\varnothing$, $R_1^{tS}=\{\be\}$, $R_2^{tS}=\{\al,\ \ga,\ \al+\be,\ \be+\ga,\ \al+\be+\ga\}$. 
For $s\in S$ we set 
\[\chi(s)=x/y,\quad \eta(s)=y/z,\quad (\chi\eta)(s)=x/z.\]
From Theorem \ref{mainthm} we have
\begin{equation}\label{211}
\tr(r_\be,V_\mu^T)=-\frac{1}{2}\sum_{v\in W}
\vep(v)(-1)^{\la v\mu,\check\al\ra}
\scp_w(v\mu-\rho).
\end{equation}
Here the partition function $\scp_w$ is given by
\[\scp_w(\nu)=
\sum_{(p,q,r)}(-1)^r,
\]
where the sum is over those $(p,q,r)\in\bn^3$ such that 
\begin{equation}\label{nu}
2p\chi+2q\eta+r(\chi+\eta)=\nu.
\end{equation}
If $\nu$ is the restriction to $S$ of $v\mu-\rho$ 
then \eqref{nu} means that
\begin{equation}\label{AC}
2p+r=A,\qquad 2q+r=C, 
\end{equation}
where $v\mu-\rho=A\al+B\be+C\ga$. 
In particular, $r\equiv A\equiv C\mod 2$. 
Setting $n=\min\{A,C\}$ we have $(-1)^r=(-1)^n$ and we find that 
\[\scp_w(v\mu-\rho)=P(n)\cdot[A-C]_{2\bz},
\]
where 
\[P(n)=(-1)^n\left(1+\gint{n}{2}\right).\]
and $[x]_{2\bz}=1$ if $x\in2\bz$, zero otherwise.

We next find those $v\in W$ for which \eqref{AC} has a solution. 
Write $v^{-1}$ as a permutation $v_1v_2v_3v_4$, sending $i\mapsto v_i$. The values of $v^{-1}\check\om_\al$ and $v^{-1}\check\om_\ga$ are determined by $v_1$ and $v_4$, respectively. 
When $v_1=4$ we have $A<0$ so there is so solution to \eqref{AC}.  Likewise, 
when $v_4=1,2$ we have $C<0$ (since $a\leq c$, in the case $v_4=2$),  so there is again no solution to \eqref{AC} .

The remaining possibilities are as follows: For each pair $(v_1, v_4)$ there are two $v$'s, 
with opposite values of $\vep(v)$. In each row the top $v$ has sign $\vep(v)=+1$ and the bottom $v$ has sign $\vep(v)=-1$, and $m=(f-3)/2$.
\[
\begin{array}{c|c|c|c|c|c}
v_1& v_4& n=\min\{A,C\}&|A-C|&v^{-1}& \la\mu,v^{-1}\check\al\ra\\
\hline
1&4&m+s&d&1234& a\\
&&&&1324&a+b\\
\hline
1&3&m&s&1423& a+b+c\\
&&&&1243&a\\
\hline
2&4&m+d&s&2314& b\\
&&&&2134&-a\\
\hline
2&3&m&d&2143& -a\\
&&&&2413&b+c\\
\hline
3&4&m-f&b+s&3124& -a-b\\
&&&&3214&-b\\
\hline
\end{array}
\]
Label the rows of this table by the pairs $(i,j)=(v_1, v_4)$. 
For each $(i,j)$ let $m_{ij},\ A_{ij},\ C_{ij}$ be the corresponding values of $m, A, C$ and set
\[
M_{ij}=P(m_{ij})\cdot[A_{ij}-C_{ij}]_{2\bz}.
\]
After further simplification, \eqref{211} becomes 
\[\begin{split}
\tr(r_\be,V_\mu^T)&=(-1)^a\left\{\ M_{13}+M_{24}-M_{14}-M_{23}-(-1)^b\cdot M_{34}\ \right\}\\
&=(-1)^a\left\{
[P(m)+P(m+d)]\cdot [s]_{2\bz}-
[P(m+s)+P(m)]\cdot [d]_{2\bz}-(-1)^bP(m-f)\cdot [s-d]_{1+2\bz}
\right\}.
\end{split}
\]
Breaking into cases for $s,d$ even/odd, we get the values for $\tr(r_\be, V_\mu^T)$ in the table above.

{\bf Remarks:\ }

1.  $V_\mu^T$ is a multiple of the regular representation when $a,b,c\in 2\bz$ and either $a$ or $c$ is in $3\bz$. For example, if $(a,b,c)=(2,4,12)$ then $
\dim V_\mu^T=24$ so $V_\mu^T$ is exactly the regular representation of $W$. 

2.\ The table of $\mu\in P_{++}$ for which $V_\mu^T$ is irreducible is as follows.
\newcommand\Tstrut{\rule{0pt}{2.4ex}}
\newcommand\Bstrut{\rule[-1.2ex]{0pt}{0pt}}
\[\begin{array} { l | l }
(a,b,c)& V_\mu^T\Bstrut\\
\hline
(1,3,1)& \varrho_2\\
(1,1,4k+1)& \vep^{k}\\
(2,1,4k+2)& \vep^{k}\otimes\varrho\\
(1,2,4k+3)& \vep^{k+1}\otimes\varrho\\
\end{array}
\]

Here, $k\in\bn$.
$\varrho$ and $\refl_2$ are the reflection and two-dimensional irreducible representations  of $S_4$, respectively. As Kostant and Gutkin discovered for $S_n$, each irrep of $S_4$ appears as the zero weight space of exactly one constituent of 
$\otimes^4\bc^4$. In fact, all but $\refl_2$ appear in infinitely many higher tensor powers as well.

3.\ We have $d=0$ exactly when $V_\mu$ is self-dual. In this case $\mu=(a,b,a)$ with $b$ odd, so $s=a$ and $f=b$. The formulas simplify to
\[\begin{split}
\tr(1^4,V_{\mu}^T)&=\tfrac{1}{2}a(ab+1)\\
\tr(211,V_{\mu}^T)&=\tfrac{1}{2}\left[(a)_2+a\cdot (b)_4\right]\\
\tr(22,V_{\mu}^T)&=\tfrac{1}{2}(-1)^{a+1}\left[a+b\cdot (a)_2\right]\\
\tr(31,V_{\mu}^T)&=(a)_3([2a+b]_{3\bn}-[b]_{3\bn})\\
\tr(4,V_{\mu}^T)&=
\begin{cases}
\ \ \ 1&\quad\text{if $a+b\in 2+4\bn$}\\
-1&\quad\text{if $a\in 2+4\bn$}\\
\ \ \ 0&\quad\text{otherwise}
\end{cases}
\end{split}
\]

\section{Irreducible zero weight spaces}\label{sec:irredzero}

It was shown already in \cite{gutkin} and \cite{kostant:eta} that every irrep   of the symmetric group $S_n$ appears as $V^T$ for some irrep of $\SL_n$. 
Since, from \cite{kostant:eta} we have
$\tr(\cox^G, V_\mu)\in\{-1,0,+1\}$, it follows that the trace of an $n$-cycle on any irrep of $S_n$ also lies in $\{-1,0,+1\}$.

When they first met, Kostant asked Lusztig if he could prove this last fact for $S_n$ directly. 
Lusztig observed that since an $n$-cycle generates its own centralizer in $S_n$, one need only find $n$ irreps of $S_n$ for which $\tr(\cox)\neq 0$. The exterior powers of the reflection representation fit the bill. 

In any irreducible Weyl group $W$ with Coxeter number $h$, it is still true that a Coxeter element generates its own centralizer and one can again find $h$ irreps (no longer exterior powers of the reflection representation) of $W$ on which $\cox$ has nonzero trace. Thus for any simple $G$ we have:
\begin{itemize}
\item[(1)] The trace of $\cox^G$ on any irrep of $G$ lies in $\{-1,0,+1\}$;
\item[(2)] The trace of $\cox$ on any irrep of $W$ lies in $\{-1,0,+1\}$.
\end{itemize}
One may ask if $(1)\Rightarrow (2)$ as it did for the symmetric group. That is, does every irrep of $W$ appear as $V^T$ for some irrep of $G$? We have seen the answer is negative, for  $G=\Sp_4$ and $G_2$. In this section we will see it is also negative for $D_4, F_4$ and $E_6$.

So the question, raised in \cite{humphreys:zerowt}, becomes: which irreps of $W$ appear as $V^T$ for some irrep $V$ of $G$? 

Theorem \ref{thm:ellreg} leads to a method for answering this question in the case that  $-1\in W$. 
This means $w_0=-1$ acts by a scalar $\pm1 $ on any irreducible representation of $W$. Since $w_0$ is an elliptic involution, Thm. \ref{thm:ellreg} implies that $V_\mu^T$ can only be irreducible if $\mu+2Q$ belongs the the $W$-orbit of $\rho+2Q$ in $P/2Q$, in which case there is $v\in W$ such that $v^{-1}\check R_2^+=\check R_\mu^+\subset \check R^+$ and 
\begin{equation}\label{ellipticinvolution}
\dim V_\mu^T=\vep(v)\prod_{\check\al\in\check R_\mu^+}\frac{\la \mu,\check\al\ra}{\la \rho,\check\al\ra}.
\end{equation}
Since $\mu\in P_{++}$,  each factor $\la\mu,\check\al\ra$ is strictly positive and increases when we add a fundamental weight to $\mu$. Hence each coset in $P/2Q$ contains only finitely many $\mu$'s for which the product in \eqref{ellipticinvolution} is the dimension of an irreducible representation of $W$. Thus one arrives at a finite list $M$ of possible $\mu's$ for which  $V_\mu^T$ can be irreducible. Computation of other character values for $\mu\in M$ can show that $V_\mu^T$ is also reducible for certain $\mu\in M$. 

For $\Spin_8$ and $F_4$ this eliminates all but the known irreducible $V_\mu^T$ as we will see. 
For $E_6$ we use the elliptic triality to do the same, up to two possible exceptions. 

\subsection{$\Spin_8$} \label{D4}
We label the Dynkin diagram of type $D_4$ as  
\[\begin{matrix}
1&2&3\\
&4&
\end{matrix}
\]
and a weight $\mu=a\om_1+b\om_2+c\om_3+d\om_4\in P$ will be written as 
\[\mu=\D{a}{b}{c}{d}.\]
We have $\rho\in Q$. It follows that the $W$-orbit of $\rho$ in $P/2Q$ lies in $Q/2Q$. Now $\mu\in Q$ if and only if $a\equiv c\equiv d\mod 2$.
We will show that the only $\mu\in Q\cap P_{++}$ with $V_\mu^T$ irreducible are the known ones (cf. \cite{reeder:zero1}):
\begin{equation}\label{D4known}
\mu=\D{1}{1}{1}{1}\qquad
\D{3}{1}{1}{1}\qquad
\D{1}{1}{3}{1}\qquad
\D{1}{1}{1}{3}\qquad
\D{1}{2}{1}{1}.
\end{equation}
These $V_\mu$ are the trivial, spherical harmonics of degree two on the three eight dimensional orthogonal representations of $G$,
and the adjoint representation, respectively.

The set 
\[\check R_\mu^+=\{\check\al\in \check R^+:\ \la \mu,\check\al\ra\in 2\bz\}
\]
depends only on the coset 
\[\mu+2P=:\left(\D{a}{b}{c}{d}\right)_{\!2P}\in P/2P.\]
The $W$-orbit of $\rho+2P$ in $P/2P$ consists of the nonzero
cosets
\begin{equation}\label{mod2P}
\left(\D{1}{1}{1}{1}\right)_{\!2P},\qquad 
\left(\D{1}{0}{1}{1}\right)_{\!2P},\qquad 
\left(\D{0}{1}{0}{0}\right)_{\!2P}.
\end{equation}
The sign character $\vep$ is trivial on the stabilizer in $W$ of each the cosets \eqref{mod2P}. It follows that the sign of $\tr(w_0,V_\mu^T)$ is constant on the fiber in $P/2Q$ above each of these cosets. 

The values of $\tr(w_0,V_\mu^T)$ are shown in the table below
\[
\begin{array} {c| c}
\mu+2P& 2^{5}\tr(w_0, V_\mu^T)\\
\hline
&\\
\left(\D{1}{1}{1}{1}\right)_{\!2P}&(a+b)(b+c)(b+d)(a+b+c+d)\\
&\\
\hline
&\\
\left(\D{1}{0}{1}{1}\right)_{\!2P}&-b(a+b+c)(a+b+d)(b+c+d)\\
&\\
\hline
&\\
\left(\D{0}{1}{0}{0}\right)_{\!2P}&-acd(a+2b+c+d)
\end{array}
\]
In each of these cosets, we find that besides the $\mu$ listed in \eqref{D4known}, there is only one other $\mu\in P_{++}$ (up to diagram symmetry) for which
\[\tr(w_0,V_\mu^T)\in\{\chi(w_0):\ \chi\in\Irr(W)\}=\{1,2,3,\pm 4,6,\pm 8\}\]
and also $|\tr(w_0, V_\mu^T)|=\dim V_\mu^T$, namely 
\[\mu=
\D{5}{1}{1}{1}.
\]
For this $\mu$ we have $\tr(w_0,V_\mu^T)=\dim(V_\mu^T)=6$. However,  $V_\mu$ is the degree-four spherical harmonics, whose zero weight space can be written down explicitly and seen to be reducible for $W$.

\subsection{$F_4$} \label{F4}
For the group $G$ of type $F_4$,  we  use our formula for $\tr(w_0,V_\mu^T)$ to determine all  irreps $V_\mu$ for which $V_\mu^T$ is irreducible for $W$. We show this occurs only for the known cases $\dim V_\mu=1,\ 26,\ 52$. 
 
We label the Dynkin diagram for $F_4$ as 
\[1\ 2\ \Rightarrow\ 3\ 4.\]
We have $P=Q$ and
the quotient $X:=P/2Q=P/2P$ is an $\bF_2$-vector space with basis $\bar\om_i:=\om_i+2P$, $i=1,\dots,4$. We write  $\mu=(a,b,c,d)$ for 
$\mu=a\om_1+b\om_2+c\om_3+d\om_4\in P$ and 
$(a,b,c,d)_2$ for the coset $\mu+2P\in X$. 

We have 
\[2P\subset Q_\ell\subset P,\]
where $Q_\ell$ is the lattice spanned by the long roots (note $Q_\ell$ was called $Q$ in section \ref{D4}). 
The $W$-action on $X$ preserves the  subspace $Y:=Q_\ell/2P$ 
and gives an isomorphism 
\[W/\{\pm 1\}\overset{\sim}\lra \GL(X,Y)=\{g\in \GL(X):\ gY=Y\}.\]
The latter group has three orbits in $X$, of sizes $1,3,12$, the latter being $X\setminus Y=W(\rho+2P)$. 
It follows that if $\mu=(a,b,c,d)$ then $\tr(w_0,V_\mu^T)\neq 0$ if and only if $c$ and $d$ are not both even. In this case there is $v\in W$ such that $\mu\in v\rho+2P$ and we have 
\[\tr(w_0,V_\mu^T)=\vep(v)\cdot 
\prod_{\check\al\in\check R_\mu^+}
\frac{\la \mu,\check\al\ra}{\la \rho,\check\al\ra}=
\frac{\vep(v)}{2^{15}\cdot 3^2\cdot 5}\cdot \prod_{\check\al\in v\check R_\mu^+}\la \mu,\check\al\ra.
\]
where we choose $v$ so that $v\check R_2^+=R_\mu^+$.
We list all positive coroots $\check \al_1,\dots, \check \al_{24}$ as linear forms $A_i(\mu)=\la\mu,\check \al_i\ra $, 
as follows. 
\[\begin{array}{l | l | l | l | }
A_1=a& A_2=b&A_3=c& A_4=d\\
A_5=a+b&A_6=b+c& A_7=c+d&\\
A_8=a+b+c&A_9=2b+c& A_{10}=b+c+d&\\
A_{11}=2b+c+d& A_{12}=a+b+c+d&A_{13}=a+2b+c&\\
A_{14}=2b+2c+d&A_{15}=a+2b+c+d&A_{16}=2a+2b+c&\\
A_{17}=a+2b+2c+d& A_{18}=2a+2b+c+d&&\\
A_{19}=a+3b+2c+d& A_{20}=2a+2b+2c+d&&\\
A_{21}=2a+3b+2c+d& &&\\
A_{22}=2a+4b+2c+d&&&\\
A_{23}=2a+4b+3c+d&&&\\
 A_{24}=2a+4b+3c+2d&&&
\end{array}
\]
We have
\[\tr(w_0,V_\mu^T)=\frac{\vep(v)}{2^{15}\cdot 3^2\cdot 5}A_{i_{1}}A_{i_2}\cdots A_{i_{10}},
\]
where 
$v\check R_2^+=\{A_{i_1},\ A_{i_2},\ \cdots A_{i_{10}}\}$.

These are shown in the table below, where the cosets in $X\setminus Y$ are labelled by $(a,b,c,d)_2$, with  $(c,d)_2\neq(0,0)_2$.
\[\def\arraystretch{1.5}
\begin{array} {c| l| l}
\mu\equiv(a,b,c,d)_2&v& 2^{15}\cdot 3^2\cdot 5\cdot\tr(w_0, V_\mu^T)\\
\hline
(1111)_2&1_W&A_{5}A_{6}A_{7}A_{11}A_{12}A_{13}A_{17}A_{18} A_{21}A_{23}\\
\hline
(1011)_2&r_1&-A_{2}A_{7}A_{8}A_{10}A_{11}A_{13}A_{17}A_{18}A_{19}A_{23}\\
(0111)_2&r_2&-A_{1}A_{6}A_{7}A_{8}A_{11}A_{15}A_{18}A_{19}A_{21}A_{23}\\
(1010)_2&r_3&-A_{2}A_{4}A_{8}A_{12}A_{13}A_{14}A_{15}A_{20}A_{21}A_{22}\\
(1101)_2&r_4&-A_{3}A_{5}A_{8}A_{9}A_{10}A_{15}A_{16}A_{17}A_{21}A_{24}
\\
\hline
(1110)_2&r_1r_3&A_{4}A_{5}A_{6}A_{10}A_{13}A_{14}A_{15}A_{19}A_{20}A_{22}\\
(1001)_2&r_1r_4&A_{2}A_{3}A_{6}A_{9}A_{12}A_{15}A_{16}A_{17}A_{19}A_{24}\\
(0010)_2&r_3r_2&A_{1}A_{2}A_{4}A_{5}A_{14}A_{17}A_{19}A_{20}A_{21}A_{22}\\
(0101)_2&r_2r_4&A_{1}A_{3}A_{9}A_{10}A_{12}A_{13}A_{16}A_{19}A_{21}A_{24}\\
\hline
(0110)_2&r_2r_1r_3&-A_{1}A_{4}A_{6}A_{8}A_{10}A_{12}A_{14}A_{17}A_{20}A_{22}\\
\hline
(0011)_2&r_3r_2r_1r_3&A_{1}A_{2}A_{5}A_{7}A_{10}A_{11}A_{12}A_{15}A_{18}A_{23}\\
\hline
(0001)_2&r_4r_3r_2r_1r_3&-A_{1}A_{2}A_{3}A_{5}A_{6}A_{8}A_{9}A_{13}A_{16}A_{24}\\
\end{array}
\]
In each of the above cosets we next find all $\mu$ for which 
\[\tr(w_0,V_\mu^T)\in\{\chi(w_0):\ \chi\in\Irr(W)\}= \{1,2,\pm 4,6,-8,9,12,-16\}.\]
In fact from all of the the twelve cosets there are only five such $\mu$. In these cases $V_\mu$ is small enough to compute $\dim V_\mu^T$, as shown in the next table.
\[\begin{array}{c| c| c}
\mu&\tr(w_0,V_\mu^T)&\dim V_\mu^T\\
\hline
(1,1,1,1)& 1 & 1\\
(2,1,1,1)& -4 & 4\\
(1,1,1,2) & 2 & 2\\
(2,2,1,1) & 4 & 228\\
(1,1,1,2) & 12 & 12
\end{array}
\]
The first three cases, where $\dim V_\mu=1,\ 52,\ 26$ are known to have irreducible zero weight spaces.
In the case  $\mu=(2,2,1,1)$, $V_\mu^T$ is clearly reducible. 
In the last case $\mu=(1,1,1,2)$ we note that $W$ has a unique 12-dimensional irreducible character $\chi_{12}$, and $\chi_{12}(\cox)=1$. On the other hand, the co-root 
$\check\be=2\check\al_1+4\check\al_2+3\check\al_3+\check\al_4$ has $\la \mu,\check\be\ra=12$, which means that $\tr(\cox,V_\mu^T)=0$, by Kostant \cite{kostant:eta} (or Prop. \ref{elltrace} above).
It follows that $V_{(2,2,1,1)}^T$ is reducible. This completes the determination of all $\mu$ for which $V_\mu^T$ is irreducible, for the case $G=F_4$.

\subsection{$E_6$}\label{E6} In this section we compute $\tr(w,V_\mu^T)$ where $w=\cox^4$ is the elliptic regular element of order three. We will see that this almost determines the irreducible zero weight spaces for $E_6$.

Since $\rho\in Q$, all  $\mu\in P_{++}$ for which $V_\mu^T\neq 0$ belong to $Q$. Hence we consider the action of $W$ on $Q/3Q$. Let $q$ be the $W$-invariant quadratic form $q:Q\to\bz$ such that $q(\al)=2$ for every $\al\in R$.

Passing to $Q/3Q$ gives a quadratic form $q_3:Q/3Q\to\bF_3$ with radical $3P/3Q$ and nondegenerate quotient $Q/3P$. 
This gives an isomorphism $\vp:W\times\{-I\}\overset{sim}\to O(Q/3P)$. 

One computes $q(\rho)=78$, so $\rho+3P$ is isotropic. 
Since $w_0\rho=-\rho$, it follows that $W$ is transitive on the set 
of nonzero isotropic vectors in $Q/3P$. Arguing as in the last step of the proof of Prop. \ref{ellregtrace}, it follows that $W$ is also transitive on the nonzero isotropic vectors in $Q/3Q$.

From this discussion and Theorem \ref{mainthm} we obtain
\begin{prop}\label{E61} For $w=\cox^4$ we have
 $\tr(w, V_\mu^T)\neq 0$ if and only if $q_3(\mu+3Q)=0$. 
In this case there exists $v\in W$ for which $v\mu\in \rho+3Q$
and we have 
\begin{equation}\label{E6triality}
\tr(w, V_\mu^T)=\frac{\vep(v)}{3^5\cdot 6^3\cdot 9}\prod_{\check\al\in \check R_\mu^+}
\la \mu,\check\al\ra.
\end{equation}
\end{prop}
To compute  $\tr(w,V_\mu^T)$ explicitly, we label the Dynkin diagram as $\E{1}{2}{3}{4}{5}{6}$ and write a weight $\mu=a\om_1+b\om_2+c\om_3+d\om_4+e\om_5+f\om_6\in P$ as 
\[\mu=\E{a}{b}{c}{d}{e}{f}.\]
As linear forms on $P$, we list the coroots as 
\[\begin{array}{l | l | l | }
A_1=a   &B_1=c&C_1=e\\
A_2=b   &B_2=f& C_2=d\\
A_3=a+b&B_3=c+f&C_3=d+e\\
A_4=b+c& B_4=b+c+d  &   C_4=c+d\\
A_5=a+b+c &B_5=b+c+d+f &C_5=c+d+e \\
A_6=b+c+f & B_6=a+b+c+d+e&    C_6=c+d+f \\
A_{7}=a+b+c+d &B_7=b+2c+d+f & C_7=b+c+d+e\\
A_{8}=a+b+c+f & B_8=a+b+c+d+e+f& C_8=c+d+e+f \\
A_{9}=a+b+c+d+f &B_9=a+b+2c+d+e+f &C_9=b+c+d+e+f \\
A_{10}=a+b+2c+d+f &B_{10}=a+2b+2c+2d+e+f &C_{10}=b+2c+d+e+f\\
A_{11}=a+2b+2c+d+f &B_{11}=a+2b+3c+2d+e+f &C_{11}=b+2c+2d+e+f\\
A_{12}=a+2b+2c+d+e+f &B_{12}=a+2b+3c+2d+e+2f &C_{12}=a+b+2c+2d+e+f
\end{array}
\]

Note that $\check R_\mu^+$ depends only on the $\bF_3$-line in $Q/3P$ containing $\mu+3P$.  
Let $X$ be the set of $q_3$-isotropic lines in $Q/3P$.
We have $|X|=40$, but we can reduce the calculation further. 

Let $\vt$ be the involution of $P/3Q$ arising from the nontrivial symmetry of the Dynkin diagram. 
Then $X=Y\sqcup Z$ where $Y=\{y_1,\dots, y_{16}\}$ consists of the $\vt$-fixed lines and $Z=\{z_1,\ \vt z_1,\ \cdots,\ z_{12},\ \vt z_{12}\}$ are the remaining lines. 

The table of $\tr(w, V_\mu^T)$ is structured as follows. 
In each isotropic line in $Q/3P$,  we give a vector $x=\sum c_i\om_i\mod 3P$ where $+, 0, -$ correspond to $c_i=+1, 0, -1\in \bF_3$. Next we give an element $v\in W$ such that $v\check R_3^+\subset \check R^+$ and $x=v\rho+3P$. Finally we give the numerator of 
$\tr(w,V_\mu^T)$ for any $\mu\in x\cap P_{++}$. If $\mu\in -x\cap P_{++}$ then $\tr(w,V_{\mu}^T)$ is the negative of the rightmost column in row $x$ evaluated at $\mu$.  If $\mu\in z_i$ then 
$\tr(w,V_{\vt\mu}^T)$ is obtained from $\tr(w,V_\mu^T)$ by interchanging $A_i\leftrightarrow C_i$. 

\newpage
\[
\begin{array} {c| c| c | c l}
x&\mu+3P=x&v& 2^{3}\cdot 3^{10}\cdot\tr(w, V_\mu^T)\\
\hline
y_1&\E{+}{+}{+}{+}{+}{+}&1_W&A_5A_6A_{10}B_4B_8B_{10}C_{10}C_6C_5\\
\hline
y_2&\E{+}{-}{-}{-}{+}{-}  &r_3 &-A_3A_6A_9B_4B_9B_{11}C_3C_6C_9\\
\hline
y_3&\E{+}{+}{-}{+}{+}{-}&r_6 &-A_4A_8A_{10}B_5B_6B_{10}C_4C_8C_{10}\\
\hline
y_4&\E{+}{+}{0}{+}{+}{+}&r_2r_1&A_7A_8A_{12}B_1B_5B_7C_7C_8C_{12}\\
\hline
y_5&\E{-}{-}{-}{-}{-}{+}. &r_4r_1&A_5A_9A_{12}B_3B_4B_7C_5C_9C_{12}\\
\hline
y_6&\E{-}{-}{+}{-}{-}{+} &r_5r_1&A_4A_8A_{10}B_5B_6B_{10}C_4C_8C_{10}\\
\hline
y_7&\E{+}{0}{+}{0}{+}{+}&r_3r_6&A_2A_8A_9B_6B_7B_{11}C_2C_8C_9\\
\hline
y_8&\E{+}{-}{+}{-}{+}{+}&r_6r_3&A_3A_4A_7B_5B_9B_{12}C_3C_4C_7\\
\hline
y_9&\E{+}{+}{-}{+}{+}{+}&r_1r_2&A_4A_9A_{11}B_3B_6B_9C_4C_9C_{11}\\
\hline
y_{10}&\E{-}{-}{0}{-}{-}{+}&r_2r_4&A_6A_7A_{11}B_1B_8B_9C_6C_7C_{11}\\
\hline
y_{11}&\E{0}{+}{0}{+}{0}{-}&r_4r_2r_3&A_1A_6A_8B_1B_{10}B_{11}C_1C_6C_8\\
\hline
y_{12}&\E{+}{0}{-}{0}{+}{-}&r_6r_3r_6&-A_2A_5A_7B_7B_8B_{12}C_2C_5C_7\\
\hline
y_{13}&\E{-}{+}{+}{+}{-}{0}&r_5r_3r_2&-A_3A_{10}A_{12}B_2B_4B_5C_3C_{10}C_{12}\\
\hline
y_{14}&\E{+}{0}{+}{0}{+}{0}&r_1r_3r_2&-A_2A_{10}A_{11}B_2B_6B_8C_2C_{10}C_{11}\\
\hline
y_{15}&\E{0}{+}{-}{+}{0}{+}&r_6r_4r_2r_3&A_1A_4A_5B_3B_{10}B_{12}C_1C_4C_5\\
\hline
y_{16}&\E{0}{0}{+}{0}{0}{0}&r_3r_6r_4r_2r_3&-A_1A_2A_3B_2B_{11}B_{12}C_1C_2C_3\\
\hline
\hline
z_1&\E{-}{-}{+}{+}{+}{+}&r_1&-A_{4}A_{7}A_{8}B_7B_9B_{10}C_{9}C_{6}{C_5}\\
\hline
z_2&\E{-}{-}{-}{+}{+}{+}&r_2&-A_{5}A_{8}A_{11}B_5B_8C_{4}C_{7}{C_{10}}C_{12}\\
\hline
z_3&\E{-}{0}{-}{-}{+}{-}&r_1r_3&A_2A_7A_8B_5B_8B_{11}C_3C_6C_{10}\\
\hline
z_4&\E{-}{-}{0}{+}{+}{-}&r_2r_6&A_8A_{11}B_1B_4B_6C_6C_9C_{10}C_{12}\\
\hline
z_5&\E{-}{+}{+}{0}{+}{0}&r_3r_2&A_3A_{11}B_2B_7B_9C_2C_7C_9C_{12}\\
\hline
z_6&\E{0}{+}{+}{-}{+}{-}&r_2r_3&A_1A_9A_{12}B_3B_5B_{11}C_3C_4B_8\\
\hline
z_7&\E{-}{-}{0}{+}{+}{+}&r_1r_2r_1&-A_6A_9A_{10}A_{12}B_1B_4B_6C_9C_{11}\\
\hline
z_8&\E{0}{-}{-}{0}{+}{0}&r_2r_3r_2&-A_1A_{10}A_{12}B_2B_7B_{10}C_2C_5C_8\\
\hline
z_9&\E{-}{0}{+}{+}{0}{-}&r_4r_1r_3&-A_2A_5A_6A_9B_3B_8B_{11}C_1C_{11}\\
\hline
z_{10}&\E{0}{+}{0}{-}{+}{+}&r_6r_2r_3&-A_1A_7A_{12}B_1B_4B_{12}C_3C_5C_6\\
\hline
z_{11}&\E{-}{0}{+}{-}{+}{+}&r_6r_1r_3&-A_2A_5A_9B_4B_6B_{12}C_3C_4C_{10}\\
\hline
z_{12}&\E{-}{0}{0}{+}{0}{+}&r_6r_4r_1r_3&A_2A_4A_7A_8B_1B_6B_{12}C_1C_{11}\\
\hline
\end{array}
\]

We turn now to irreducible zero weight spaces for $E_6$.

As an abstract group, we have $W\simeq \SO_5(3)$, via the  twisted mapping $\vep\vp$. The Steinberg representation $81_p$ of $\SO_5(3)$ and its twist $81_p'=\vep\otimes 81_p$ are the only irreducible representations of $W$ whose character vanishes on $w$. We conclude a dichotomy:

\begin{prop}\label{E6irr} If $V_\mu^T$ is irreducible for $W$ 
then exactly one of the following holds. 
\begin{itemize}
\item[\rm (i)] $q_3(\mu)=0$ and $\tr(w,V_\mu^T)\neq 0$.
\item[\rm(ii)] $q_3(\mu)\neq 0$ and $V_\mu^T\simeq 81_p$ or $81_p'$. 
\end{itemize}
\end{prop}
I do not know if case (ii) ever occurs, but it seems unlikely. 

Up to duality, there are five known representations $V_\mu$  for which $V_\mu^T$ is irreducible. These are small representations (cf. \cite{reeder:zero1}). 
\begin{equation}\label{E6small}
\begin{array}{c|c|c}
\mu& V_\mu^T& \tr(w,V_\mu^T)\\
\hline
\E{1}{1}{1}{1}{1}{1}&1_p&\ 1\\
\E{1}{1}{1}{1}{1}{2}&6_p &\!\!\!-3\\ 
\E{2}{1}{1}{1}{2}{1}& 20_p & \ 2\\
\E{4}{1}{1}{1}{1}{1}\qquad \E{1}{1}{1}{1}{4}{1}& 24_p & \ 6\\
\E{2}{2}{1}{1}{1}{1}\qquad \E{1}{1}{1}{2}{2}{1}& 64_p &\!\!\!-8\\
\end{array}
\end{equation}
The representation $64_p$ is another Steinberg representation, via the isomorphism $W\simeq O_6^-(2)$. However, because $-1\notin W$, there is no analogue of Prop. \ref{E6irr} for $m=2$.

Using the table of $\tr(w,V_\mu^T)$ and the method used for $D_4$ and $F_4$, we find that \eqref{E6small} is the complete list  of  irreducible zero weight spaces with $\tr(w,V_\mu^T)\neq 0$.  We conclude 
\begin{prop} The only irreducible representations of $W(E_6)$ afforded by a zero weight space of the compact Lie group $E_6$ are $1_p, 6_p, 20_p, 24_p, 64_p$ and possibly 
$81_p, 81_p'$. 
\end{prop}

\def\noopsort#1{}
\providecommand{\bysame}{\leavevmode\hbox to3em{\hrulefill}\thinspace}

\end{document}